\documentclass[11pt,a4paper]{amsart}
\usepackage{amsmath, amsthm, amsfonts, amssymb}
\usepackage[bookmarks=false]{hyperref}

\newtheorem{letterthm}{Theorem}

\newtheorem{lettercor}[letterthm]{Corollary}

\newtheorem{theo}{Theorem}[section]
\newtheorem{cor}[theo]{Corollary}
\newtheorem{lem}[theo]{Lemma}
\newtheorem{prop}[theo]{Proposition}

\theoremstyle{definition}

\newtheorem{example}[theo]{Example}

\newtheorem{df}[theo]{Definition}
\newtheorem*{claim}{Claim}
\newtheorem*{definition}{Definition}

\newcommand{\R}{\mathbb{R}}
\newcommand{\C}{\mathbb{C}}
\newcommand{\Z}{\mathbb{Z}}
\newcommand{\F}{\mathbb{F}}

\newcommand{\N}{\mathbb{N}}

\newcommand{\Id}{\mathord{\text{\rm Id}}}
\newcommand{\Ad}{\operatorname{Ad}}

\newcommand{\Tr}{\operatorname{Tr}}
\newcommand{\alg}{\text{\rm alg}}

\newcommand{\cb}{\text{\rm cb}}
\newcommand{\SO}{\operatorname{SO}}
\newcommand{\SU}{\operatorname{SU}}
\newcommand{\HNN}{\operatorname{HNN}}
\newcommand{\Ind}{\operatorname{Ind}}
\newcommand{\core}{\mathord{\text{\rm c}}}
\newcommand{\rL}{\mathord{\text{\rm L}}}
\newcommand{\red}{\text{\rm red}}
\newcommand{\cN}{\mathcal{N}}
\newcommand{\cU}{\mathcal{U}}
\newcommand{\recht}{\rightarrow}
\newcommand{\QHreg}{\mathcal{QH}_{\text{\rm reg}}}
\newcommand{\B}{\operatorname{B}}
\newcommand{\M}{\operatorname{M}}
\newcommand{\cF}{\mathcal{F}}
\newcommand{\cS}{\mathcal{S}}
\newcommand{\cG}{\mathcal{G}}
\newcommand{\cH}{\mathcal{H}}
\newcommand{\cV}{\mathcal{V}}
\newcommand{\cW}{\mathcal{W}}
\newcommand{\cR}{\mathcal{R}}
\newcommand{\cO}{\mathcal{O}}
\newcommand{\cK}{\mathcal{K}}
\newcommand{\cZ}{\mathcal{Z}}
\newcommand{\cE}{\mathcal{E}}
\newcommand{\cT}{\mathcal{T}}
\newcommand{\cI}{\mathcal{I}}
\newcommand{\cM}{\mathcal{M}}
\newcommand{\cB}{\mathcal{B}}
\newcommand{\actson}{\curvearrowright}
\newcommand{\vphi}{\varphi}
\newcommand{\diag}{\operatorname{diag}}
\newcommand{\eps}{\varepsilon}
\newcommand{\vphih}{\widehat{\varphi}}
\newcommand{\si}{\sigma}
\newcommand{\lspan}{\operatorname{span}}
\newcommand{\ot}{\otimes}
\newcommand{\ovt}{\mathbin{\overline{\otimes}}}
\newcommand{\cHtil}{\widetilde{\mathcal{H}}}
\newcommand{\minim}{\text{\rm min}}
\newcommand{\Om}{\Omega}
\newcommand{\rd}{\mathord{\text{\rm d}}}
\newcommand{\cA}{\mathcal{A}}
\newcommand{\cP}{\mathcal{P}}
\newcommand{\cJ}{\mathcal{J}}
\newcommand{\dpr}{^{\prime\prime}}
\newcommand{\Gammah}{\widehat{\Gamma}}
\newcommand{\cC}{\mathcal{C}}
\newcommand{\om}{\omega}
\newcommand{\Ytil}{\widetilde{Y}}
\newcommand{\Ker}{\operatorname{Ker}}

\newcommand{\one}{\mathbf{1}}

\newcommand{\tI}{$\text{\rm I}$}
\newcommand{\tIinfty}{$\text{\rm I}_\infty$}
\newcommand{\tII}{$\text{\rm II}$}
\newcommand{\tIIone}{$\text{\rm II}_1$}
\newcommand{\tIII}{$\text{\rm III}$}
\newcommand{\tIIinfty}{$\text{\rm II}_\infty$}

\begin{document}

\title[Type III factors with unique Cartan decomposition]{Type III factors with unique Cartan decomposition}

\author[Cyril Houdayer]{Cyril Houdayer*}
\address{CNRS-ENS Lyon \\
UMPA UMR 5669 \\
69364 Lyon cedex 7 \\
France}
\email{cyril.houdayer@ens-lyon.fr}
\thanks{*Research partially supported by ANR grant AGORA NT09-461407.}

\author[Stefaan Vaes]{Stefaan Vaes**}
\address{KU Leuven \\ Department of Mathematics \\ Celestijnenlaan 200B \\ B-3001 Leuven \\ Belgium}
\email{stefaan.vaes@wis.kuleuven.be}
\thanks{**Research partially supported by ERC Starting Grant VNALG-200749, Research Programme G.0639.11 of the Research Foundation - Flanders (FWO) and K.U.Leuven BOF research grant OT/08/032.}


\begin{abstract}
We prove that for any free ergodic nonsingular nonamenable action $\Gamma \actson (X,\mu)$ of all $\Gamma$ in a large class of groups including all hyperbolic groups, the associated group measure space factor $\rL^\infty(X) \rtimes \Gamma$ has $\rL^\infty(X)$ as its unique Cartan subalgebra, up to unitary conjugacy. This generalizes the probability measure preserving case that was established in \cite{PV12}. We also prove primeness and indecomposability results for such crossed products, for the corresponding orbit equivalence relations and for arbitrary amalgamated free products $M_1 *_B M_2$ over a subalgebra $B$ of type \tI.
\end{abstract}

\maketitle

\section{Introduction and main results}

A \emph{Cartan subalgebra} in a (separable) factor $M$ is a maximal abelian $*$-subalgebra $A \subset M$ such that there exists a normal faithful conditional expectation $E : M \recht A$ and such that the normalizer $\cN_M(A) = \{u \in \cU(M) \mid u A u^* = A\}$ generates $M$.

The presence of a Cartan subalgebra in a factor $M$ amounts to decomposing $M = \rL_\Om(\cR)$ as the von Neumann algebra associated with a countable ergodic nonsingular equivalence relation $\cR$ on a standard measure space and a scalar $2$-cocycle $\Om$ on $\cR$ (see \cite{feldman-moore}). In particular, $\rL^\infty(X)$ is a Cartan subalgebra in the group measure space construction $\rL^\infty(X) \rtimes \Gamma$ whenever $\Gamma \actson (X,\mu)$ is a free ergodic nonsingular action. The classification problem of group measure space factors therefore splits into two different questions: uniqueness of the Cartan subalgebra and classification of orbit equivalence relations.
We focus on the first of the above questions and prove the uniqueness of Cartan subalgebras in a family of type \tIII\ factors.

Since \cite{connes-feldman-weiss,CJ81}, it is known that two Cartan subalgebras in an amenable factor are always conjugate by an automorphism, but that this can fail in nonamenable factors. The first uniqueness results in the nonamenable case were obtained in \cite{Po01}, where it is shown that \emph{rigid} Cartan subalgebras of an arbitrary type \tIIone\ crossed product $\rL^\infty(X) \rtimes \F_n$ with the free group on $2 \leq n \leq \infty$ generators, must be unitarily conjugate to $\rL^\infty(X)$.
The paper \cite{Po01} moreover introduced a whole arsenal of new methods to study the structure of \tIIone\ factors and was the starting point of Popa's deformation/rigidity theory (see \cite{Po06c,Va10a} for surveys).

The results in \cite{Po01} led to the conjecture that all crossed products $\rL^\infty(X) \rtimes \F_n$ by free ergodic probability measure preserving (pmp) actions of the free group could have a unique Cartan subalgebra up to unitary conjugacy. For \emph{profinite} actions $\F_n \actson (X,\mu)$, it was shown in \cite{OP07} that this is indeed the case, providing in particular the first \tIIone\ factors with a unique Cartan subalgebra up to unitary conjugacy. The key points of \cite{OP07} were the \emph{weak compactness} for arbitrary Cartan subalgebras of \tIIone\ factors with the complete metric approximation property, and the \emph{malleable deformation} of \cite{Po06b} for arbitrary crossed products by the free group.

In the period 2008-2011, the uniqueness of the Cartan subalgebra in a crossed product $\rL^\infty(X) \rtimes \Gamma$ by a \emph{profinite} free ergodic pmp action $\Gamma \actson (X,\mu)$ was shown for more and more groups $\Gamma$, by using weaker and weaker forms of deformation for $\Gamma$. The article \cite{OP08} covered groups $\Gamma$ with the complete metric approximation property and a proper $1$-cocycle into a nonamenable representation, by using the deformation of \cite{Pe06} associated with a densely defined derivation. In \cite{Oz10}, it was shown that in order to prove weak compactness of Cartan subalgebras, the complete metric approximation property may be replaced by weak amenability in the sense of \cite{cowling-haagerup}. Then, in \cite{CS11}, a weak type of malleable deformation could be constructed from a map $c : \Gamma \recht K$ that only coarsely satisfies the $1$-cocycle relation with respect to a unitary representation $\pi : \Gamma \recht \cU(K)$, in the sense that $\sup_{k \in \Gamma} \|c(gkh) - \pi(g) c(k)\| < \infty$ for all $g,h \in \Gamma$. Following \cite{CS11}, $\Gamma$ is said to belong to the class $\QHreg$ if $c$ can be chosen proper and into a representation $\pi$ that is weakly contained in the regular representation. So more precisely, \cite{CS11} showed that $\rL^\infty(X) \rtimes \Gamma$ has a unique Cartan subalgebra up to unitary conjugacy for all profinite free ergodic pmp actions of all nonamenable, weakly amenable groups $\Gamma$ in the class $\QHreg$. Because of \cite{Oz07}, this covers all nonelementary hyperbolic groups. This result was then extended in \cite{CSU11} to cover as well direct products of nonamenable, weakly amenable group $\Gamma$ in the class $\QHreg$.

In \cite{PV11}, the above conjecture was entirely solved and it was proven that for \emph{arbitrary} free ergodic pmp actions of the free group $\F_n$, the \tIIone\ factor $\rL^\infty(X) \rtimes \F_n$ has a unique Cartan subalgebra up to unitary conjugacy. In the terminology of \cite{PV11}, this means that the free group $\F_n$, $2 \leq n \leq \infty$, is $\cC$-rigid. More precisely, $\cC$-rigidity was shown in \cite{PV11} for all weakly amenable groups that admit a proper $1$-cocycle into a nonamenable representation. The key point of \cite{PV11} was to prove that Cartan subalgebras of \emph{arbitrary} crossed products $\rL^\infty(X) \rtimes \Gamma$ by weakly amenable groups $\Gamma$ satisfy a \emph{relative} (w.r.t.\ $\rL^\infty(X)$) \emph{weak compactness property}.
In the followup paper \cite{PV12}, also nonamenable weakly amenable groups in the class $\QHreg$, and their direct products, were shown to be $\cC$-rigid.

All the results stated so far focused on probability measure preserving actions. The following is our main result, proving the uniqueness of the Cartan subalgebra for arbitrary \emph{nonsingular} actions of weakly amenable groups in the class $\QHreg$. As such, we obtain in particular the first type \tIII\ factors with a unique Cartan subalgebra. Note that the class of weakly amenable groups in $\QHreg$ contains all hyperbolic groups, all lattices in a connected noncompact rank one simple Lie group with finite center, and all limit groups in the sense of Sela (see \cite[Lemma 2.4]{PV12}).

\begin{letterthm}\label{thmA}
Let $\Gamma_1,\ldots,\Gamma_n$ be \emph{weakly amenable} groups in the class $\QHreg$ and put $\Gamma = \Gamma_1 \times \cdots \times \Gamma_n$. Let $\Gamma \curvearrowright (X, \mu)$ be any free ergodic nonsingular action on the standard measure space $(X,\mu)$. Denote by $M = \rL^\infty(X) \rtimes \Gamma$ the group measure space factor. Then at least one of the following statements holds.
\begin{itemize}
\item There exists an $i \in \{1,\ldots,n\}$ such that $\rL^\infty(X) \rtimes \Gamma_i$ is amenable.
\item $\rL^\infty(X)$ is the unique Cartan subalgebra of $M$ up to unitary conjugacy.
\end{itemize}
\end{letterthm}

For probability measure preserving actions, the amenability of $\rL^\infty(X) \rtimes \Gamma_i$ is equivalent with the amenability of the group $\Gamma_i$. For general nonsingular free actions, this is no longer the case. Then, the amenability of $\rL^\infty(X) \rtimes \Gamma_i$ is equivalent with the amenability of the action $\Gamma_i \actson (X,\mu)$ in the sense of Zimmer, see \cite[Theorem 2.4]{Zi76}.

The proof of Theorem \ref{thmA} relies heavily on the results in \cite{PV12}. Given a Cartan subalgebra $A \subset M = \rL^\infty(X) \rtimes \Gamma$, we apply the modular theory of Tomita, Takesaki and Connes to finally obtain a tracial situation where the key Theorem~3.1 of \cite{PV12} can be applied. In order to descend back to the initial factor $M$, we then need a type \tIII\ version of Popa's intertwining-by-bimodules theorem in \cite{Po01} (see Theorem \ref{cartan}).

With the following explicit examples of nonsingular actions of the free group $\F_2$, we obtain factors with unique Cartan subalgebra, having any possible type and any possible flow of weights in the type \tIII$_0$ case.

\begin{lettercor}\label{corA}
Let $\Gamma = \F_2$ be the free group with two free generators $a,b \in \Gamma$. Denote by $\pi : \Gamma \recht \Z$ the group homomorphism given by $\pi(a) = 1$ and $\pi(b) = 0$. Choose a free probability measure preserving action $\Gamma \actson (X,\mu)$ such that $\Ker \pi$ acts ergodically (e.g.\ a Bernoulli action of $\Gamma$). Choose an arbitrary free ergodic nonsingular action $\Z \actson (Y,\eta)$ on a nonatomic standard measure space. Define the action $\Gamma \actson X \times Y$ given by $g \cdot (x,y) = (g \cdot x, \pi(g) \cdot y)$. Put $M = \rL^\infty(X \times Y) \rtimes \Gamma$.

Then $M$ is a factor and $\rL^\infty(X \times Y) \subset M$ is the unique Cartan subalgebra of $M$ up to unitary conjugacy. The factor $M$ has the same type and the same flow of weights as the amenable factor $\rL^\infty(Y) \rtimes \Z$.
\end{lettercor}

Related to the existence and uniqueness problem for Cartan subalgebras are other (in)decom\-posability phenomena for von Neumann algebras. The second part of the paper focuses on tensor product decompositions of factors and \emph{primeness}. A factor $M$ is called prime if $M$ cannot be written as the tensor product of two non type \tI\ factors.

In \cite{Ge96}, as an application of Voiculescu's free entropy theory, the free group factors were shown to be prime. It was proven in \cite{Oz03} that for all groups $\Gamma$ in his \emph{class $\cS$}, and thus in particular for all hyperbolic groups, the associated group von Neumann algebra $\rL \Gamma$ is \emph{solid,} meaning that $A' \cap \rL \Gamma$ is amenable whenever $A \subset \rL \Gamma$ is a diffuse subalgebra. Note that nonamenable solid \tIIone\ factors are obviously prime. We also mention here that a group belongs to class $\cS$ if and only if it is exact and belongs to the above class $\QHreg$ of \cite{CS11} (see e.g.\ \cite[Proposition 2.7]{PV12}).

In \cite{ozawa-kurosh}, it was proven that an arbitrary crossed product $M = \rL^\infty(X) \rtimes \Gamma$ of a group $\Gamma$ in class $\cS$ and a probability measure preserving action $\Gamma \actson (X,\mu)$, is \emph{semisolid}, meaning that $P' \cap M$ is amenable whenever $P \subset M$ is a von Neumann subalgebra of type \tIIone. In particular, every orbit equivalence relation $\cR(\Gamma \actson X)$ associated with a free pmp action of a group in class $\cS$, is \emph{indecomposable}, meaning that it cannot be written as a direct product of two equivalence relations of type \tIIone. This last result had been proven before for hyperbolic groups $\Gamma$ in \cite{adams}. We extend these results to nonsingular actions. Before giving precise formulations, we introduce some terminology.

For a (possibly nonunital) von Neumann subalgebra $P \subset M$, we denote by $1_P$ the unit of $P$, which is a nonzero projection of $M$. We say that the von Neumann subalgebra $P \subset 1_P M1_P$ is \emph{with expectation} if there exists a faithful normal conditional expectation $E_P : 1_P M 1_P \to P$. We also use Popa's notion of \emph{intertwining-by-bimodules,} denoted by $\preceq$, see Section \ref{sec.intertwine}.

\begin{letterthm}\label{thmB}
Let $\Gamma$ be any \emph{exact} group in the class $\QHreg$ and $\Gamma \curvearrowright (X, \mu)$ any nonsingular action on a standard measure space. Denote by $M = \rL^\infty(X) \rtimes \Gamma$ the group measure space von Neumann algebra. Let $e \in M$ be a projection and $P \subset e M e$ any von Neumann subalgebra with expectation. Then at least one of the following holds:
\begin{itemize}
\item There exists a nonzero projection $p \in P$ such that $pPp \preceq_M \rL^\infty(X)$.
\item The relative commutant $P' \cap e M e$ is amenable.
\end{itemize}
In particular, any nonamenable subfactor $Q \subset e M e$ with expectation is prime.
\end{letterthm}

For probability measure preserving actions, Theorem~\ref{thmB} had already been proven in \cite[Theorem 4.6]{ozawa-kurosh} and \cite[Theorem 3.2]{CS11}. An independent proof for the general case of nonsingular actions was given in \cite{isono}.

Following \cite[Definition 2.1]{adams}, a nonsingular countable equivalence relation $\cR$ on a standard measure space $(X, \mu)$ is called
\begin{itemize}
\item \emph{recurrent} if for every Borel subset $\cW \subset X$ with $\mu(\cW) > 0$, and for $\mu$-almost every $x \in \cW$, the intersection $\cW \cap \{y : (x, y) \in \cR\}$ is infinite. This is equivalent to saying that the von Neumann algebra $\rL(\cR)$ of the equivalence relation $\cR$ has no type \tI\ direct summand \cite{feldman-moore}.
\item \emph{decomposable} if $\cR$ splits as a direct product $\cS_1 \times \cS_2$ of recurrent nonsingular equivalence relations.
\end{itemize}

The following result is an equivalence relation version of Theorem \ref{thmB}.

\begin{letterthm}\label{thmC}
Let $\Gamma$ be \emph{any} group in the class $\QHreg$ and $\Gamma \curvearrowright (X, \mu)$ any free nonsingular action on a standard measure space. Let $\cV \subset X$ be a non-negligible subset. Every nonamenable recurrent subequivalence relation of $\cR(\Gamma \curvearrowright X)_{|\cV}$ is indecomposable.
\end{letterthm}

The special case of Theorem~\ref{thmC} where $\Gamma$ is a hyperbolic group and where one only considers the entire orbit equivalence relation $\cR(\Gamma \actson X)$, rather than arbitrary recurrent subequivalence relations, was obtained in \cite[Theorem 6.1]{adams}. For exact groups in the class $\QHreg$, Theorem \ref{thmC} is an immediate corollary of Theorem \ref{thmB}. Since it is not clear whether there even exist nonexact groups in the class $\QHreg$, it might seem not so interesting to make the effort to prove Theorem \ref{thmC} in its full generality. But as we will see, the proof of Theorem \ref{thmB} uses the very deep fact that exact C$^*$-algebras are locally reflexive. The main reason to give a direct proof for Theorem \ref{thmC} precisely is to avoid the use of local reflexivity.

Another source of prime von Neumann algebras is given by \emph{(amalgamated) free products.} In \cite{ozawa-kurosh}, it was shown that the free product $M = M_1 * M_2$ of finite factors $M_i$ containing a dense exact C$^*$-subalgebra, is always prime, unless $M_1 \cong \C$ or $M_2 \cong \C$.

In \cite{IPP05}, a systematic Bass-Serre theory for amalgamated free products $M = M_1 *_B M_2$ of finite von Neumann algebras, was developed. In particular, \cite{IPP05} introduced a malleable deformation (in the sense of \cite{popa-malleable1,Po06c}) for such an amalgamated free product. A combination of this malleable deformation and the spectral gap rigidity methods of \cite{Po06a,Po06b}, then lead in \cite{CH08} to a general primeness result for amalgamated free product von Neumann algebras $M = M_1 \ast_B M_2$ over a type \tI\ subalgebra $B$.

In \cite{CH08}, the von Neumann algebras $M_i$ need no longer be tracial. More precisely, \cite[Theorem 5.2]{CH08} states that whenever an amalgamated free product $M= M_1 *_B M_2$ with $B$ of type \tI\ is a nonamenable factor and if $M_1 \neq B \neq M_2$, then $M$ must be prime. It was however overlooked in \cite{CH08} that for general inclusions $B \subset M_i$, the condition $M_1 \neq B \neq M_2$ is too weak to avoid ``trivialities'' in which a corner of $M$ equals a corner of one of the $M_i$'s. We settle this technical issue here. Moreover our stronger assumption on $B \subset M_i$ automatically implies that $M$ is a nonamenable factor.

We say that an inclusion $A \subset P$ has a \emph{trivial corner} if there exists a nonzero projection $p \in A' \cap P$ such that $Ap = pPp$. We obtain the following general result.

\begin{letterthm}\label{thmD}
For $i = 1, 2$, let $B \subset M_i$ be an inclusion of von Neumann algebras with separable predual. Assume that for all $i =1,2$, the inclusion $B \subset M_i$ is without trivial corner and with expectation. Assume that $B$ is a type \tI\ von Neumann algebra and that $M_1' \cap M_2' \cap B = \C$. Then the amalgamated free product $M_1 \ast_B M_2$ is a nonamenable prime factor.
\end{letterthm}

Theorem \ref{thmD} says in particular that an amalgamated free product $M_1 *_B M_2$ over a type \tI\ von Neumann algebra $B$ and inclusions $B \subset M_i$ without trivial corner, is always a factor. This generalizes earlier factoriality results of Ueda (see \cite{ueda-pacific,ueda-cartan-II,ueda-transactions}). Very recently, in \cite[Theorem 4.3]{ueda12}, Ueda independently established a more general factoriality theorem for amalgamated free products. For details, we refer to the discussion preceding Theorem \ref{amalgamated}.

Recall that for an inclusion of countable discrete groups $\Sigma < H$ and an injective group homomorphism $\theta : \Sigma \to H$, the HNN extension $\HNN(\Sigma, H, \theta)$ is the group generated by a copy of $H$ and an extra letter $t$ (called the stable letter), such that $t \sigma t^{-1} = \theta (\sigma)$, for all $\sigma \in \Sigma$.

\begin{lettercor}\label{corE}
Let $\Gamma \curvearrowright (X, \mu)$ be a nonsingular free ergodic action of a countable discrete group on a standard measure space. Assume that one of the following holds.
\begin{enumerate}
\item $\Gamma = \Gamma_1 \ast_\Sigma \Gamma_2$ with $\Sigma$ a finite group and $\Gamma_i \curvearrowright (X, \mu)$ a recurrent action for $i = 1, 2$.
\item $\Gamma = \HNN(\Sigma, H, \theta)$ with $\Sigma$ a finite group and $H \curvearrowright (X, \mu)$ a recurrent action.
\end{enumerate}
Then $\rL^\infty(X) \rtimes \Gamma$ is a nonamenable prime factor and the orbit equivalence relation $\cR(\Gamma \curvearrowright X)$ is indecomposable.
\end{lettercor}

Both in Theorem \ref{thmB} and Theorem \ref{thmC}, one needs to assume the nonamenability of the crossed product or equivalence relation before concluding primeness or indecomposability. This assumption is crucial since also nonamenable groups admit amenable nonsingular actions. As we will see in the proof of Corollary~\ref{corE}, the recurrence of $\Gamma_i \actson X$, resp.\ $H \actson X$, in the hypotheses of Corollary \ref{corE} ensures a priori that the action $\Gamma \actson X$ is nonamenable.

\section{Preliminaries}

As explained above, a von Neumann subalgebra $N \subset M$ is said to be \emph{with expectation,} if there exists a normal faithful conditional expectation $E : M \recht N$.

\subsection{Weakly amenable groups in the class $\QHreg$}

\begin{definition}[\cite{CS11}]
A countably infinite group $\Gamma$ belongs to the class $\QHreg$ if $\Gamma$ admits a unitary representation $\pi : \Gamma \to \cU(H)$ that is weakly contained in the regular representation and a proper map $c : \Gamma \to H$ satisfying $\sup_{x \in \Gamma} \|c(g x h) - \pi_g(c(x))\| < \infty$ for all $g, h \in \Gamma$.
\end{definition}

Chifan-Sinclair's class $\QHreg$ is closely related to Ozawa's class $\cS$ \cite{ozawa-kurosh}. Recall from \cite[Section 4]{ozawa-kurosh} that a countably infinite group $\Gamma$ belongs to the class $\cS$ if the left-right translation action of $\Gamma \times \Gamma$ on the Stone-Cech boundary $\partial^\beta \Gamma$ is topologically amenable. Thus, a group in the class $\cS$ is necessarily exact. As observed by Ozawa, the class of exact groups in $\QHreg$ coincides with the class $\cS$ (see \cite{CS11,CSU11}, as well as \cite[Proposition 2.7]{PV12}). So,
$$\{\mbox{exact groups}\} \cap \QHreg = \cS \; .$$

Recall from \cite{cowling-haagerup} that a countable group $\Gamma$ is \emph{weakly amenable} if there exists a sequence of finitely supported functions $\varphi_n : \Gamma \to \C$ such that $\varphi_n \to 1$ pointwise and such that the maps $u_g \mapsto \varphi_n(g) u_g$ extend linearly to completely bounded maps ${\rm m}_{\varphi_n} : \rL(\Gamma) \to \rL(\Gamma)$ with $\sup_n \|{\rm m}_{\varphi_n}\|_{\cb} < \infty$. A weakly amenable group is necessarily exact.

\subsection{Two elementary lemmas}

For the convenience of the reader, we provide detailed proofs for the following rather standard lemmas.

\begin{lem}\label{lem.equiv}
Let $(\cM,\Tr)$ be a von Neumann algebra with a normal semifinite faithful trace $\Tr$. Assume that $\cB \subset \cM$ is a maximal abelian von Neumann subalgebra such that $\Tr_{|\cB}$ is semifinite.
Then every finite trace projection in $\cM$ is equivalent with a projection in $\cB$.
\end{lem}
\begin{proof}
Fix a finite trace projection $p \in \cM$. We first claim that there exists a nonzero central projection $z \in \cZ(\cM)$ and a finite trace projection $q \in \cB$ such that $pz \precsim q$. If $\Tr(1) < \infty$, we take $z=1=q$. So, assume that $\Tr(1) = +\infty$ and assume that such $z$ and $q$ do not exist. By comparability of projections, we get that $q \precsim p$ for all finite trace projections $q \in \cB$. It follows in particular that $\Tr(q) \leq \Tr(p)$ for all finite trace projections $q \in \cB$. Since $\Tr_{|\cB}$ is semifinite, we reached the contradiction that $\Tr(1) \leq \Tr(p) < \infty$.

By Zorn's lemma, take a maximal orthogonal family of central projections $(z_j)_{j \in J}$ in $\cZ(\cM)$ with the property that every $p z_j$ is subequivalent with a finite trace projection in $\cB$. Put $z = 1-\sum_{j \in J} z_j$. If $z \neq 0$, we apply the previous paragraph to $\cM z$, $pz$ and $\cB z$, and contradict the maximality of the family $(z_j)_{j \in J}$. So, $\sum_{j \in J} z_j = 1$.

Choose finite trace projections $q_j \in \cB z_j$ such that $p z_j \precsim q_j$. Put $q = \sum_{j \in J} q_j$. By construction, we have $p \precsim q$. Because $\cZ(\cM) \subset \cB$, we also have that $q \in \cB$. Since the central portions $q z_j = q_j$ of $q$ are finite, it follows that $q$ is a finite projection in $\cB$. So $q \cM q$ is a finite von Neumann algebra with maximal abelian von Neumann subalgebra $\cB q$. By \cite[Corollary F.8]{BO08}, every projection in $q \cM q$ is equivalent with a projection in $\cB q$. So also $p$ is equivalent with a projection in $\cB q \subset \cB$.
\end{proof}

\begin{lem}\label{lem.masa-typeI}
Let $P$ be a von Neumann algebra and $B \subset P$ an abelian subalgebra with expectation.
\begin{enumerate}
\item If $P$ is of type \tI, there exists a family of projections $p_i \in B' \cap P$ such that $\sum_i p_i = 1$ and such that every $p_i$ is an abelian projection in $P$.
\item If $p \in P$ is an abelian projection and $B \subset P$ is maximal abelian, there exists a partial isometry $v \in P$ such that $v^* v = p$, $vv^* \in B$ and $v P v^* = B vv^*$.
\end{enumerate}
\end{lem}
\begin{proof}
(1) Let $(p_i)$ be a maximal orthogonal family of projections in $B' \cap P$ that are abelian as projections in $P$. Put $q = 1 - \sum_i p_i$ and assume that $q \neq 0$. Replacing $B$ by $Bq$ and replacing $P$ by $qPq$, we actually only have to prove that $B' \cap P$ contains a nonzero projection $p$ that is abelian in $P$.

Choose a normal state $\tau$ on $B$ and choose a faithful normal conditional expectation $E : P \recht B$. Denote by $e \in B$ the support projection of $\tau$. Define the faithful normal state $\vphi$ on $ePe$ given by $\vphi = \tau \circ E$. Choose a faithful normal semifinite trace $\Tr$ on $P$. Take $\Delta \in \rL^1(ePe)^+$ such that $\vphi(x) = \Tr(\Delta x)$ for all $x \in ePe$. Since $\vphi(x b) = \vphi(b x)$ for all $x \in ePe$ and all $b \in Be$, it follows that $\Delta$ is affiliated with $(B' \cap P)e$. Take $\eps > 0$ small enough such that $r := \one_{[\eps,+\infty)}(\Delta)$ is nonzero. So, $r$ is a nonzero projection in $B' \cap P$. Since $\Tr(\Delta) = \vphi(e) = 1$, we also get that $\Tr(r) < +\infty$. So, the von Neumann algebra $rPr$ is of finite type \tI\ and $Br \subset rPr$ is an abelian subalgebra. Choose $Br \subset D \subset rPr$ such that $D \subset rPr$ is maximal abelian. Since $rPr$ is of type \tI, it contains a nonzero abelian projection $f$. Since $rPr$ is finite and $D \subset rPr$ is maximal abelian, it follows from \cite[Corollary F.8]{BO08} that $f$ is equivalent with a projection $p$ in $D$. Then also $p$ is an abelian projection. Since $D \subset B' \cap P$, we have found the required abelian projection in $B' \cap P$.

(2) Fix an abelian projection $p \in P$. Denote by $z \in \cZ(P)$ the central support of $p$. Replacing $P$ by $Pz$ and $B$ by $Bz$, we may assume that $P$ is of type \tI\ and that the central support of $p$ equals $1$. Take a maximal family of orthogonal central projections $z_i \in \cZ(P)$ with the property that $p z_i$ is equivalent with a projection in $B z_i$. Write $e = 1 - \sum_i z_i$. We prove that $e=0$. If $e \neq 0$, also $ep \neq 0$ and statement (1) provides a projection $q \in B' \cap P = B$ such that $q$ is abelian in $P$ and $q e p \neq 0$. The right support of $qep$ belongs to $p e P p = p e \cZ(P)$ and is equivalent with the left support of $q e p$, which belongs to $qPq e = \cZ(P) qe \subset Be$. This contradicts the maximality of the family $(z_i)$ and we conclude that $\sum_i z_i = 1$. It follows that $p$ is equivalent with a projection in $B$. Since $p$ is an abelian projection and $B \subset P$ is maximal abelian, any partial isometry $v \in P$ with $v^* v = p$ and $vv^* \in B$, also satisfies $v P v^* = B vv^*$.
\end{proof}

\subsection{Variations on Popa's intertwining techniques} \label{sec.intertwine}

In \cite{popa-malleable1,Po01}, Popa discovered the following powerful method to unitarily conjugate subalgebras of a finite von Neumann algebra. Let $(M, \tau)$ be a tracial von Neumann algebra and $A\subset 1_A M 1_A$, $B \subset1_B M 1_B$ von Neumann subalgebras. By \cite[Corollary 2.3]{popa-malleable1} and \cite[Theorem A.1]{Po01}, the following three statements are equivalent:
\begin{itemize}
\item There exist $n \geq 1$, a possibly nonunital normal $\ast$-homomorphism $\pi : A \to \M_n(B)$ and a nonzero partial isometry $v \in \M_{1, n}(1_A M 1_B)$ such that $a v = v \pi(a)$ for all $a \in A$.

\item There is no net of unitaries $(w_i)$ in $A$ such that
$$\lim_{i \to \infty} \|E_B(x^* w_i y)\|_2 = 0, \forall x, y \in 1_A M 1_B \; .$$

\item There exists a nonzero $A$-$B$-subbimodule of $1_A \rL^2(M) 1_B$ that has finite dimension as a right $B$-module.
\end{itemize}
If one of the previous equivalent conditions is satisfied, we say that $A$ {\it embeds into} $B$ {\it inside} $M$ and write $A \preceq_M B$.

Let $M$ be any von Neumann algebra and denote by $(M, H, J, \mathfrak P)$ a standard form for $M$. Recall from \cite[Chapter IX]{takesakiII} that any vector $\xi \in H$ has a unique polar decomposition $\xi = u |\xi|$ with  $|\xi| \in \mathfrak P$ and $u \in M$ a partial isometry such that $u^*u = [M' |\xi|]$ and $uu^* = [M' \xi]$. Moreover, whenever $N \subset 1_N M 1_N$ is a von Neumann subalgebra and $\xi \in H$ is
a vector that satisfies $x \xi = J x^* J \xi$ for all $x \in N$, then the equality $x u = u x$ holds inside $M$, for all $x \in N$.

We need the following straightforward generalization of Popa's Intertwining Theorem. Its proof is almost identical to the proof of \cite[Theorem 2.1, Corollary 2.3]{popa-malleable1}, but we include the details for the sake of completeness. A further generalization can be found in \cite[Proposition 3.1]{ueda12}.

\begin{theo}\label{intertwining-general}
Let $M$ be any $\sigma$-finite von Neumann algebra. Let $A \subset 1_A M 1_A$ and $B \subset 1_B M 1_B$ be von Neumann subalgebras such that $B$ is finite and with expectation $E_B : 1_B M 1_B \recht B$. The following are equivalent.
\begin{enumerate}
\item There exist $n \geq 1$, a possibly nonunital normal $\ast$-homomorphism $\pi : A \to \M_n(B)$ and a nonzero partial isometry $v \in \M_{1, n}(1_A M 1_B)$ such that $a v = v \pi(a)$ for all $a \in A$.

\item There is no net of unitaries $(w_i)$ in $\cU(A)$ such that $E_B(x^* w_i y) \to 0$ $\ast$-strongly for all $x, y \in 1_A M 1_B$.
\end{enumerate}
\end{theo}

\begin{proof}
We only need to prove that $(2) \Longrightarrow (1)$. Fix a faithful normal tracial state $\tau$ on $B$ and define the faithful normal state $\varphi$ on $1_B M 1_B$ given by $\varphi(x) = \tau (E_B(x))$ for all $x \in 1_BM1_B$. Extend $\vphi$ to a faithful normal positive functional on $M$ by choosing a faithful normal state on $1_B^\perp M 1_B^\perp$.
Denote $\cH := J 1_B J \, \rL^2(M, \varphi)$. If $z$ denotes the central support of $1_B$ in $M$, observe that $Mz$ is faithfully represented on $\cH$. Denote by $e_B$ the orthogonal projection of $\cH$ onto $\rL^2(B,\tau) \subset \cH$. Consider the von Neumann algebra $\cN := \B(\cH) \cap (J B J)'$. Since $\tau$ is a faithful trace on $B$, there is a canonical faithful normal semifinite trace $\Tr$ on $\cN$ satisfying the formula $\Tr(T T^*) = \tau(T^* T)$ for all bounded right $B$-linear maps $T : \rL^2(B, \tau) \to \cH$. Observe that $e_B \in \cN$ with $\Tr(e_B) = 1$. Since $Mz$ acts faithfully on $\cH$, we will regard $Mz$ as a von Neumann subalgebra of $\cN$.

On bounded subsets of $B$, the strong$^*$ topology coincides with the $\|\,\cdot\,\|_2$-topology. So,
following the lines of \cite[Corollary 2.3]{popa-malleable1}, if $(2)$ holds, we can find $\delta > 0$ and a finite subset $\cF \subset 1_A M 1_B$ such that $\sum_{x, y \in \cF} \|E_B(x^* w y)\|_2^2 \geq \delta$ for all $w \in \cU(A)$. Let $\xi = \sum_{x \in \cF} x e_B x^* \in \cN_+$ and note that $\Tr(\xi) = \sum_{x \in \cF} \tau(E_B(x^*x)) < \infty$. Denote by $\eta \in \cN_+$ the unique element of minimal $\|\cdot\|_{2, \Tr}$-norm in the weak closure of the convex hull of $\{w \xi w^* : w \in \cU(A)\}$. We then have that $\eta \in A' \cap 1_A \cN_+ 1_A$ and $\Tr(\eta) \leq \Tr(\xi) < \infty$. A simple computation shows that for all $w \in \cU(A)$,
$$\sum_{y \in \cF} \Tr(e_B y^* w^* \xi w y e_B) = \sum_{x, y \in \cF} \tau(E_B(x^* w y)^*E_B(x^* w y)) \geq \delta.$$
This shows that $\sum_{y \in \cF} \Tr(e_B y^* \eta y e_B) \geq \delta$ and thus $\eta \neq 0$. By taking a suitable spectral projection of $\eta$, we can find a nonzero projection $p \in A' \cap 1_A \cN 1_A$ such that $\Tr(p) < \infty$.

The nonzero $A$-$B$-bimodule $p \cH$ has thus finite dimension over $B$. We may then find a nonzero $A$-$B$-subbimodule $\cK \subset p \cH$ which is finitely generated as a right $B$-module (see e.g.\ \cite[Lemma A.1]{vaes-bourbaki-popa}). Let $n \geq 1$ and $q \in \M_n(B)$ a nonzero projection such that there exists a right $B$-module isomorphism $\psi : \cK_B \to (q \rL^2(B)^{\oplus n})_B$. Let $\pi : A \to q \M_n(B) q$ be the unital $\ast$-homomorphism such that $\psi(a \zeta) = \pi(a) \psi(\zeta)$ for all $a \in A$ and all $\zeta \in \cK$. Take $\xi_i \in \cK$ such that $\psi(\xi_i) = q (0, \dots, 0, 1_B, 0, \dots, 0)$. With $\xi = (\xi_i) \in \M_{1, n}(\C) \otimes \cK$, a simple computation shows that $a \xi = J\pi(a)^* J \xi$ for all $a \in A$. Taking the polar decomposition of the vector $\xi$ in a standard representation of $\M_n(M)$, we find a nonzero partial isometry $v \in \M_{1,n}(1_A M 1_B)$ such that $a v = v \pi(a)$ for all $a \in A$.
\end{proof}

\begin{df}
Let $M$ be any $\sigma$-finite von Neumann algebra. Let $A \subset 1_A M1_A$ and $B \subset 1_B M 1_B$ be von Neumann subalgebras such that $B$ is finite and with expectation. We say that $A$ \emph{ embeds into} $B$ \emph{inside} $M$ and denote $A \preceq_M B$ if one of the equivalent conditions of Theorem \ref{intertwining-general} is satisfied. Note that this forces $A$ to have a finite direct summand.
\end{df}

For Cartan subalgebras $A$ and $B$, the embedding $A \preceq_M B$ is equivalent with the unitary conjugacy of $A$ and $B$. This fundamental result was proven by Popa in \cite{Po01} in the type \tIIone\ setting. His proofs carries over almost verbatim to the type \tIII\ setting. Again we provide the details for the sake of completeness.

\begin{theo}[Popa, \cite{Po01}]\label{cartan}
Let $M$ be any $\sigma$-finite von Neumann algebra and $A, B \subset M$ unital maximal abelian $\ast$-subalgebras with expectation. Consider the following assertions.
\begin{enumerate}
\item $A \preceq_M B$.
\item There exists a nonzero partial isometry $w \in M$ such that $w^*w \in A$, $ww^* \in B$ and $w A w^* = Bww^*$.
\item There exists a unitary $u \in M$ such that $u A u^* = B$.
\end{enumerate}
Then $(1) \Longleftrightarrow (2) \Longleftarrow (3)$. If moreover $A, B \subset M$ are Cartan subalgebras and $M$ is a factor, then $(1) \Longleftrightarrow (2) \Longleftrightarrow (3)$.
\end{theo}

\begin{proof}
We first prove that $(1) \Longrightarrow (2)$. Since $A \preceq_M B$, we can take a nonzero partial isometry $v \in \M_{1, n}(M)$, a projection $q \in \M_n(B)$ and a unital $\ast$-homomorphism $\pi : A \recht q \M_n(B) q$ such that $a v = v \pi(a)$ for all $a \in A$. Since $A$ and $B$ are abelian, we may replace $q$ by an equivalent projection in $\M_n(B)$ and assume that $q = \diag(q_1,\ldots,q_n)$ and $\pi(A) \subset \diag(B q_1,\ldots,B q_n)$. Choosing a nonzero component of $v$, we end up with a nonzero partial isometry $v \in M$, a projection $q \in B$ and a unital $\ast$-homomorphism $\pi : A \to Bq$ such that $a v = v \pi(a)$ for all $a \in A$.

We have that $vv^* \in A' \cap M = A$. We also conclude that $v^*v \in \pi(A)' \cap qMq$ is an abelian projection, since
$$v^*v (\pi(A)' \cap q M q) v^*v = v^* ((A vv^*)' \cap vv^* M vv^*) v = v^* A v.$$
By restricting the given conditional expectation $E_B : M \recht B$, we get that $B q \subset \pi(A)' \cap qMq$ is a maximal abelian subalgebra with expectation. Since $v^*v$ is an abelian projection in $\pi(A)' \cap qMq$, Lemma \ref{lem.masa-typeI} provides a partial isometry $u \in \pi(A)' \cap qMq$ such that $u^* u = v^* v$ and $u u^* \in B$. Writing $w = u v^*$, we get that $w^* w = v v^* \in A$, that $w w^* = u u^* \in B$ and $w A w^* \subset B ww^*$. Since $A$ is maximal abelian and $B$ is abelian, we get the equality $w A w^* = B ww^*$.

If moreover $A, B \subset M$ are Cartan subalgebras and $M$ is a factor, we prove that $(2) \Longrightarrow (3)$. If $M$ is a type \tIII\ factor, there exist isometries $w_1, w_2 \in M$ such that $w_1 A w_1^* = A w^*w$ and $w_2 B w_2^* = B ww^*$. Then $u = w_2^* w w_1 \in \cU(M)$ is a unitary such that $u A u^* = B$.
Finally assume that $M$ is a semifinite factor with semifinite faithful normal trace $\Tr$. Choose a nonzero projection $p \in A w^*w$ with $\Tr(p) < \infty$. Write $q = w p w^*$. Choose projections $(p_n)$ in $A$ and $(q_n)$ in $B$ such that $\sum_n p_n = 1 = \sum_n q_n$ and $\Tr(p_n) = \Tr(q_n) \leq \Tr(p)$ for all $n$. We can then take partial isometries $u_n \in M$ such that $u_n u_n^* = p_n$, $u_n^* u_n \in A p$ and $u_n A u_n^* = A p_n$. We next take partial isometries $v_n \in M$ such that $v_n^* v_n = w u_n^* u_n w^*$, $v_n v_n^* = q_n$ and $v_n B v_n^* = B q_n$. One checks that the formula $u = \sum_n v_n w u_n^*$ defines a unitary in $M$ satisfying $u A u^* = B$.
\end{proof}

The following lemma is a straightforward variant of \cite[Corollary F.14]{BO08} and will be used in the proof of Theorem~\ref{thmB}.

\begin{lem}\label{intertwining-abelian}
Let $M$ be a von Neumann algebra with separable predual. Let $A \subset 1_A M 1_A$ and $B \subset M$ be unital von Neumann subalgebras with expectation. Assume that $B$ is abelian and that for every nonzero projection $p \in A$, we have $p A p \npreceq_M B$. Then there exists a diffuse abelian $\ast$-subalgebra $D \subset A$ with expectation such that $D \npreceq_M B$.
\end{lem}

\begin{proof}
Let $z \in \cZ(A)$ be the unique projection such that $Az$ is of type \tI\ and $A(1_A - z)$ has no type \tI\ direct summand. Fix a unital maximal abelian $\ast$-subalgebra with expectation $D_0 \subset A z$ and a sequence of orthogonal projections $e_n \in D_0$ such that $D_0 e_n = e_n (Az) e_n$ for every $n$ and $\sum_n e_n = z$. Since $e_n A e_n \npreceq_M B$, we get that $D_0 \npreceq_M B$.

Denote $N = A(1_A - z)$. By \cite[Theorem 11.1]{haagerup-stormer}, we can choose a faithful normal state $\varphi$ on $N$ such that the centralizer $N^\varphi$ is of type \tIIone. By \cite[Theorem IX.4.2]{takesakiII}, every von Neumann subalgebra of $N^\varphi$ is with expectation as a subalgebra of $N$.
Following the lines of \cite[Corollary F.14]{BO08}, we construct a unital diffuse abelian $\ast$-subalgebra $D_1 \subset N^\vphi$ such that $D_1 \npreceq_M B$.

Fix a faithful normal trace $\tau$ on $B$ and let $\psi = \tau \circ E_B$ be the corresponding faithful normal state on $M$. Denote by $\|\cdot\|_2$ the $\rL^2$-norm associated with the state $\psi$ on $M$. Fix a countable subset $\{x_i : i \in \N\} \subset (1_A-z)M$ that is dense in $(1_A-z)M$ with respect to the norm $\|\cdot\|_2$.

By induction, we construct an increasing sequence $(Q_n)$ of unital abelian finite dimensional $\ast$-subalgebras of $N^\varphi$ and unitaries $w_n \in \cU(Q_n)$ such that $\|E_B(x_i^* w_n x_j)\|_2 < n^{-1}$ for all $1 \leq i, j \leq n$. Let $Q_1 = \C (1_A - z)$ and assume that $Q_1, \dots, Q_n$ and $w_1, \dots, w_n$ have been constructed. Denote by $(p_k)_{1 \leq k \leq \dim Q_n}$ the minimal projections of $Q_n$. Since $p_k N^\varphi p_k$ is still of type \tIIone, we have that $p_k N^\varphi p_k \npreceq_M B$ and therefore we can find a unitary $v_k \in p_k N^\varphi p_k$ such that $\|E_B(x^*_i v_k x_j)\|_2 < ((n + 1) \dim Q_n)^{-1}$ for all $1 \leq i, j \leq n + 1$. By the spectral theorem, we can uniformly approximate $v_k$ by a unitary in $p_k N^\varphi p_k$ with finite spectrum. So, we may assume that $v_k$ has finite spectrum. We define $w_{n + 1} = \sum_k v_k$ and define $Q_{n + 1}$ as the abelian $\ast$-algebra generated by $Q_n$ and $w_{n + 1}$. By construction, $Q_{n+1}$ and $w_{n+1}$ satisfy the required properties.

By construction, we have for all $i,j \in \N$ that $\lim_n \|E_B(x_i^* w_n x_j)\|_2 = 0$. Since the $(x_i)_{i \in \N}$ span a $\|\cdot\|_2$-dense subset of $(1_A-z)M$ and since $E_B$ is contractive in $\|\cdot\|_2$-norm, we get that $\lim_n \|E_B(x^* w_n y)\|_2 = 0$ for all $x, y \in (1_A-z)M$. This shows that $D_1 := \bigvee_n Q_n$ satisfies $D_1 \npreceq_M B$. Therefore $D := D_0 \oplus D_1 \subset 1_A M 1_A$ is a unital diffuse abelian $\ast$-subalgebra with expectation such that $D \npreceq_M B$.
\end{proof}

We finally prove the following semifinite version of \cite[Proposition 2.6]{vaes-cohomology}.

\begin{lem}\label{lem.intertwine-entirely}
Let $(\cB,\Tr)$ be a von Neumann algebra with a normal semifinite faithful trace. Let $\Gamma$ be a countable group, $\Gamma \actson \cB$ a trace preserving action and $q \in \cB$ a finite trace projection. Put $\cM := q(\cB \rtimes \Gamma) q$ and let $p \in \cM$ be a nonzero projection. Assume that $\cA \subset p \cM p$ is a von Neumann subalgebra with normalizer $\cP := \cN_{p\cM p}(\cA)\dpr$.

For every finite subset $\cF \subset \Gamma$, we denote by $P_\cF$ the orthogonal projection of $\rL^2(\cM,\Tr)$ onto the closed linear span of $\{q b u_g q \mid b \in \cB , g \in \cF\}$. We denote by $\|\,\cdot\,\|_2$ the $2$-norm on $\rL^2(\cM,\Tr)$.
\begin{enumerate}
\item The set of projections $\cJ := \{e \in \cA' \cap p \cM p \mid \cA e \npreceq_{\cM} q \cB q\}$ is directed and attains its maximum in a projection $z$ that belongs to $\cZ(\cP)$.
\item There exists a sequence of unitaries $(w_k)$ in $\cU(\cA z)$ such that $\lim_k \|P_\cF(w_k)\|_2 = 0$ for all finite subsets $\cF \subset \Gamma$.
\item For every $\eps > 0$, there exists a finite subset $\cF \subset \Gamma$ such that $\|a - P_\cF(a)\|_2 < \eps$ for all $a$ in the unit ball of $\cA (p-z)$.
\end{enumerate}
\end{lem}
\begin{proof}
(1) Assume that $(e_i)_{i \in I}$ is a family of projections in $\cA' \cap p \cM p$ and put $e= \vee_{i \in I} e_i$. To prove that $\cJ$ is directed and attains its maximum, we must show that if $e \not\in \cJ$, then $e_i \not\in \cJ$ for one of the $i \in I$. If $e \not\in \cJ$, we have $\cA e \preceq_\cM q \cB q$ and we find a nonzero partial isometry $v \in \M_{1,n}(\C) \ot e \cM$ and a normal $*$-homomorphism $\vphi : \cA \recht \M_{n}(\C) \ot q \cB q$ such that $a v = v \vphi(a)$ for all $a \in \cA$. Take $i \in I$ such that $e_i v \neq 0$ and denote by $w \in \M_{1,n}(\C) \ot e_i \cM$ the polar part of $e_i v$. Since $a w = w \vphi(a)$ for all $a \in \cA$, it follows that $\cA e_i \preceq_\cM q \cB q$ and hence $e_i \not\in \cJ$.

Denote by $z \in \cJ$ the maximum of $\cJ$. If $e \in \cJ$ and $u \in \cN_{p \cM p}(\cA)$, it is immediate that $u e u^* \in \cJ$. So, $z$ commutes with $\cN_{p\cM p}(\cA)$ and it follows that $z \in \cZ(\cP)$.

(2) By construction $\cA z \npreceq_\cM q \cB q$. Put $\cN = \cB \rtimes \Gamma$. Denote by $z_1 \in \cZ(\cB)$ the central support of the projection $q \in \cB$. We claim that $\cA z \npreceq_{r \cN r} r \cB r$ whenever $r \in \cB$ is a finite trace projection with $q \leq r \leq z_1$. Assume the contrary and take a nonzero partial isometry $v \in \M_{1,n}(\C) \ot z \cN r$ and a normal $*$-homomorphism $\vphi : \cA z \recht \M_n(\C) \ot r \cB r$ such that $a v = v \vphi(a)$ for all $a \in \cA z$.

Since $r \leq z_1$ and since $z_1$ is the central support of $q$ in $\cB$, we can take a central projection $z_2 \in \cZ(\cB)$ such that $r z_2$ is arbitrarily close to $r$ and $r z_2 = r_1 + \cdots + r_m$ with all the $r_i$ being projections in $\cB$ that are subequivalent with $q$. Since $v = v (1 \ot r)$, we choose $z_2$ such that $v(1 \ot z_2) \neq 0$. We take a partial isometry $w \in \M_{1,m}(\C) \ot \cB$ such that $ww^* = r z_2$ and $w^* w \leq 1 \ot q$. Define $\psi : \cA z \recht \M_{nm}(\C) \ot q \cB q$ given by $\psi(a) = (1 \ot w^*) \vphi(a) (1 \ot w)$. Since $v (1 \ot w) \neq 0$ and $a v(1 \ot w) = v(1 \ot w) \psi(a)$ for all $a \in \cA z$, it follows that $\cA z \preceq_{\cM} q \cB q$, contrary to our choice of $z$. This proves the claim.

Choose a finite subset $\cF \subset \Gamma$ and choose $\eps > 0$. To prove (2), it suffices to construct a unitary $w \in \cU(\cA z)$ such that $\|P_\cF(w)\|_2 < \eps$. We still denote by $z_1 \in \cZ(\cB)$ the central support of the projection $q \in \cB$. Define the finite trace projection $r \in \cB$ given by
$$r = \Bigl(q \vee \bigvee_{g \in \cF} \si_g(q)\Bigr) z_1 \; .$$
By the claim above, $\cA z \npreceq_{r \cN r} r \cB r$. For every $g \in \cF$, the element $q u_g^* z_1$ belongs to $r \cN r$. So we can find a unitary $w \in \cU(\cA z)$ such that
$$\bigl\| E_{r \cB r} (w \, q u_g^* z_1) \bigr\|_2 < \frac{\eps}{|\cF|} \quad\text{for all}\;\; g \in \cF \; .$$
Denote by $E_\cB : \cN \recht \cB$ the canonical conditional expectation. Note that for all $x \in \cM = q \cN q$, we have $x = z_1 x q$, so that for all $g \in \cF$,
$$E_{\cB}(x u_g^*) = E_{\cB}(z_1 x q \, u_g^*) = z_1 \, E_\cB(x q u_g^*) = E_\cB(x q u_g^*) \, z_1 = E_\cB(x \, q u_g^* z_1) = E_{r \cB r}(x \, q u_g^* z_1) \; .$$
Since for all $x \in \cM$, we have
$$P_\cF(x) = \sum_{g \in \cF} E_\cB(x u_g^*) \, u_g \; ,$$
we get that
$$\|P_\cF(w)\|_2 \leq \sum_{g \in \cF} \|E_\cB(w u_g^*)\|_2 = \sum_{g \in \cF} \|E_{r \cB r}(w \, q u_g^* z_1)\|_2 < \eps \; .$$
This concludes the proof of (2).

(3) By construction, we have $\cA e \preceq_\cM q \cB q$ for every nonzero projection $e \in \cA' \cap p \cM p$ satisfying $e \leq p-z$. Choose $\eps > 0$. Take an orthogonal family of nonzero projections $e_1,\ldots,e_k \in \cA' \cap p \cM p$ such that
\begin{itemize}
\item $e_i \leq p-z$ for every $i=1,\ldots,k$ and $e = e_1 + \cdots + e_k$ satisfies $\|(p-z) - e\|_2 < \eps / 3$~;
\item for every $i =1,\ldots,k$, there exists a partial isometry $v_i \in \M_{1,n_i}(\C) \ot \cM$ and a normal $*$-homomorphism $\vphi_i : \cA \recht \M_{n_i}(\C) \ot q \cB q$ such that $v_i v_i^* = e_i$ and $a v_i = v_i \vphi_i(a)$ for all $a \in \cA$.
\end{itemize}
Put $n = n_1+\cdots+n_k$ and define $\vphi : \cA \recht \M_n(\C) \ot q\cB q$ by putting together the $\vphi_i$ diagonally. Similarly define the partial isometry $v \in \M_{1,n}(\C) \ot \cM$ such that $vv^* = e$ and $a v = v \vphi(a)$ for all $a \in \cA$. By the Kaplansky density theorem, choose $v_0 \in \M_{1,n}(\C) \ot q(\cB \rtimes_\alg \Gamma)q$ such that $\|v_0\| \leq 1$ and $\|v - v_0\|_2 < \eps / 3$. Define the finite subset $\cG \subset \Gamma$ such that $v_0$ belongs to the linear span of $\{e_{1i} \ot q b u_g q \mid i=1,\ldots,n, b \in \cB , g \in \cG\}$. Put $\cF = \cG \cG^{-1}$.

Take $a$ in the unit ball of $\cA (p-z)$. We must prove that $\|a - P_\cF(a)\|_2 < \eps$. Write $a = a(p-z-e) + ae$ and note that $\|a (p-z-e)\|_2 \leq \|p-z-e\|_2 < \eps/3$. Since $ae = v \vphi(a) v^*$, it follows that $ae$ lies at distance less than $2\eps/3$ of $v_0 \vphi(a) v_0^*$. The latter belongs to the image of $P_\cF$. So, $a$ lies at distance less than $\eps$ of an element in the image of $P_\cF$, i.e.\ $\|a - P_\cF(a)\|_2 < \eps$.
\end{proof}

\subsection{Connes-Takesaki's noncommutative flow of weights}\label{flow}

Let $(M, \varphi)$ be a von Neumann algebra together with a faithful normal state. Denote by $M^\varphi$ the centralizer of $\varphi$ and by $M \rtimes_\varphi \R$ the {\it continuous core} of $M$, i.e.\ the crossed product of $M$ with the modular automorphism group $(\sigma_t^\vphi)_{t \in \R}$ associated with the faithful normal state $\varphi$. We have a canonical embedding $\pi_\vphi : M \recht M \rtimes_\vphi \R$ and a canonical group of unitaries $(\lambda_\vphi(s))_{s \in \R}$ in $M \rtimes_\vphi \R$ such that
$$\pi_\vphi(\si^\vphi_s(x)) = \lambda_\vphi(s) \, \pi_\vphi(x) \, \lambda_\vphi(s)^* \quad\text{for all}\;\; x \in M, s \in \R \; .$$
The unitaries $(\lambda_\vphi(s))_{s \in \R}$ generate a copy of $\rL(\R)$ inside $M \rtimes_\vphi \R$.

We denote by $\vphih$ the \emph{dual weight} on $M \rtimes_\vphi \R$ (see \cite[Definition X.1.16]{takesakiII}), which is a normal semifinite faithful weight on $M \rtimes_\vphi \R$ whose modular automorphism group $(\sigma_t^{\vphih})_{t \in \R}$ satisfies
$$
\sigma_t^{\widehat{\varphi}}(\pi_\varphi(x)) = \pi_\varphi(\sigma_t^\varphi(x)) \;\;\text{for all}\;\; x \in M \quad\text{and}\quad
\sigma_t^{\widehat{\varphi}}(\lambda_\varphi(s)) = \lambda_\varphi(s) \;\;\text{for all}\;\; s \in \R \;  .
$$
We denote by $(\theta^\varphi_t)_{t \in \R}$ the \emph{dual action} on $M \rtimes_\vphi \R$, given by
$$\theta^\vphi_t(\pi_\vphi(x)) = \pi_\vphi(x) \;\;\text{for all}\;\; x \in M \quad\text{and}\quad \theta^\vphi_t(\lambda_\vphi(s)) = \exp(i t s) \lambda_\vphi(s) \;\;\text{for all}\;\; s \in \R \; .$$
Denote by $h_\varphi$ the unique nonsingular positive selfadjoint operator affiliated with $\rL(\R) \subset M \rtimes_\vphi \R$ such that $h_\varphi^{is} = \lambda_\varphi(s)$ for all $s \in \R$. Then $\Tr_\varphi := \widehat{\varphi}(h_\varphi^{-1} \cdot)$ is a semifinite faithful normal trace on $M \rtimes_\varphi \R$ and the dual action $\theta^\varphi$ {\it scales} the trace $\Tr_\varphi$:
\begin{equation*}
\Tr_\varphi \circ \theta^\varphi_t = \exp(t) \Tr_\varphi \quad\text{for all}\;\; t \in \R \; .
\end{equation*}
Note that $\Tr_\vphi$ is semifinite on $\rL(\R) \subset M \rtimes_\vphi \R$. Moreover, the canonical faithful normal conditional expectation $E_{\rL(\R)} : M \rtimes_\varphi \R \to \rL(\R)$ defined by $E_{\rL(\R)}(x \lambda_\varphi(s)) = \varphi(x) \lambda_\varphi(s)$ preserves the trace $\Tr_\varphi$, that is,
\begin{equation*}
\Tr_\varphi \circ E_{\rL(\R)} = \Tr_\varphi \; .
\end{equation*}

Because of Connes's Radon-Nikodym cocycle theorem (see \cite[Theorem VIII.3.3]{takesakiII}), the semifinite von Neumann algebra $M \rtimes_\vphi \R$, together with its trace $\Tr_\vphi$ and trace-scaling action $\theta^\vphi$, ``does not depend'' on the choice of $\vphi$ in the following precise sense. If $\psi$ is another faithful normal state on $M$, there is a canonical surjective $*$-homomorphism
$\Pi_{\psi,\vphi} : M \rtimes_\vphi \R \recht M \rtimes_\psi \R$ such that $\Pi_{\psi,\vphi} \circ \pi_\vphi = \pi_\psi$, $\Tr_\psi \circ \Pi_{\psi,\vphi} = \Tr_\vphi$ and $\Pi_{\psi,\vphi} \circ \theta^\vphi_t = \theta^\psi_t \circ \Pi_{\psi,\vphi}$. Note however that $\Pi_{\psi,\vphi}$ does not map the subalgebra $\rL(\R) \subset M \rtimes_\vphi \R$ onto the subalgebra $\rL(\R) \subset M \rtimes_\psi \R$.

Altogether we can abstractly consider the \emph{continuous core} $(\core(M),\theta,\Tr)$, where $\core(M)$ is a von Neumann algebra with a faithful normal semifinite trace $\Tr$, $\theta$ is a trace-scaling action of $\R$ on $(\core(M),\Tr)$ and $\core(M)$ contains a copy of $M$. Whenever $\vphi$ is a faithful normal state on $M$, we get a canonical surjective $\ast$-homomorphism $\Pi_\vphi : M \rtimes_\vphi \R \recht \core(M)$ such that
$$\Pi_{\varphi} \circ \theta^\varphi = \theta \circ \Pi_{\varphi} \;\; ,\quad \Tr_\varphi = \Tr \circ \Pi_{\varphi} \;\; , \quad \Pi_{\varphi}(\pi_\varphi(x)) = x \;\; \forall x \in M \; .$$
A more functorial construction of the continuous core, known as the {\it noncommutative flow of weights} can be given, see \cite{connes73,connestak,falcone}.

By Takesaki's duality theorem \cite[Theorem X.2.3]{takesakiII}, we have that $\core(M) \rtimes_{\theta} \R \cong M \ovt \B(\rL^2(\R))$. In particular, $M$ is amenable if and only if $\core(M)$ is amenable.

Whenever $\Gamma \curvearrowright (X, \mu)$ is a nonsingular action on a standard measure space, define the \emph{Maharam extension} (see \cite{maharam}) $\Gamma \curvearrowright (X \times \R, m)$ by
$$g \cdot (x, t) = \left( gx, t + \log\left( \frac{{\rm d} \mu \circ g^{-1}}{{\rm d} \mu}(x)\right)\right),$$
where ${\rm d}m = {\rm d} \mu \times \exp(t){\rm d}t$. It is easy to see that the action $\Gamma \curvearrowright X \times \R$ preserves the infinite measure $m$ and that we moreover have that
$$\core(\rL^\infty(X) \rtimes \Gamma) = \rL^\infty(X \times \R) \rtimes \Gamma.$$
If $P \subset 1_P M 1_P$ is a von Neumann subalgebra with expectation, we have a canonical trace preserving inclusion $\core(P) \subset 1_P \core(M) 1_P$.

\section{Proof of Theorem~\ref{thmA} and Corollary~\ref{corA}}

We start by proving Theorem \ref{thmA} in the infinite measure preserving case. More precisely, we deduce the following result from its finite measure preserving counterpart proven in
\cite[Theorem 3.1]{PV12}.

\begin{theo}\label{newkey}
Let $\Gamma_1,\ldots,\Gamma_n$ be weakly amenable groups in the class $\QHreg$ and put $\Gamma = \Gamma_1 \times \cdots \times \Gamma_n$. Let $(\cB,\Tr)$ be an amenable von Neumann algebra equipped with a normal semifinite faithful trace $\Tr$. Assume that $q \in \cB$ is a finite trace projection and put $\cM = q(\cB \rtimes \Gamma)q$. Assume that $\cA \subset \cM$ is a von Neumann subalgebra such that $\cA$ is amenable and $\cA \subset \cM$ is regular, i.e.\ $\cN_\cM(\cA)\dpr = \cM$. Then at least one of the following statements holds.
\begin{itemize}
\item There exists an $i \in \{1,\ldots,n\}$ such that $\cB \rtimes \Gamma_i$ has an amenable direct summand.
\item For every nonzero projection $p \in \cA' \cap \cM$, we have $\cA p \preceq_{\cM} q \cB q$.
\end{itemize}
\end{theo}
\begin{proof}
Assume that for all $i \in \{1,\ldots,n\}$, the von Neumann algebra $\cB \rtimes \Gamma_i$ has no amenable direct summand. We will prove that for every nonzero projection $p \in \cA' \cap \cM$, we have $\cA p \preceq_{\cM} q \cB q$.

For every $i \in \{1,\ldots,n\}$, we denote by $\Gammah_i$ the product of all $\Gamma_j$ with $j \neq i$. Put $\cN = \cB \rtimes \Gamma$ and $\cN_i := \cB \rtimes \Gammah_i$. So we identify $\cN = \cN_i \rtimes \Gamma_i$ and $\cM = q (\cN_i \rtimes \Gamma_i) q$. We also write $\cM_i = q \cN_i q$ and we denote by $E_{\cM_i} : \cM \recht \cM _i$ the unique trace preserving conditional expectation.

Fix $i \in \{1,\ldots,n\}$. For every finite subset $\cF \subset \Gamma_i$, denote by $P^i_\cF$ the orthogonal projection of $\rL^2(\cM)$ onto the closed linear span of $\{q x u_g q \mid x \in \cN_i, g \in \cF \}$. Since $\cN_\cM(\cA)\dpr = \cM$, Lemma \ref{lem.intertwine-entirely} provides a central projection $z \in \cZ(\cM)$ and a sequence of unitaries $w_k \in \cU(\cA z)$ such that
\begin{itemize}
\item $\lim_k \|P^i_\cF(w_k)\|_2 = 0$ for all finite subsets $\cF \subset \Gamma_i$~;
\item for every $\eps > 0$, there exists a finite subset $\cF' \subset \Gamma_i$ such that $\|a - P^i_{\cF'}(a)\|_2 < \eps$ for all $a$ in the unit ball of $\cA(q-z)$.
\end{itemize}
We will prove that $z = 0$. We will actually show that if $z \neq 0$, then $\cB \rtimes \Gamma_i$ has an amenable direct summand.

So, assume that $z \neq 0$. Denote by $\Delta : \cM \recht \cM \ovt \rL \Gamma_i$ the \emph{dual coaction}, i.e.\ the normal unital trace preserving $*$-homomorphism given by
$$\Delta : \cM \recht \cM \ovt \rL \Gamma_i : \Delta(q x u_g q) = q x u_g q \ot u_g \quad\text{for all}\;\; x \in \cN_i, g \in \Gamma_i \; .$$
For every finite subset $\cF \subset \Gamma_i$, denote by $Q_\cF$ the orthogonal projection of $\rL^2(\rL \Gamma_i)$ onto the closed linear span of $\{u_g \mid g \in \cF\}$. Note that $(1 \ot Q_\cF)\Delta(x) = \Delta(P^i_\cF(x))$ for all $x \in \cM$ and all finite subsets $\cF \subset \Gamma_i$. Since $\Delta$ is $\|\,\cdot\, \|_2$-preserving, it follows that
$$\lim_k \|(1 \ot Q_\cF)(\Delta(w_k))\|_2 = 0 \quad\text{for all finite subsets}\;\; \cF \subset \Gamma_i \; .$$
This implies that $\Delta(\cA z) \npreceq_{\cM \ovt \rL \Gamma_i} \cM \ot 1$. Note that $\cA z$ is amenable. Put $z_1 = \Delta(z) \in \cM \ovt \rL \Gamma_i$. Because $z \in \cZ(\cM)$, the normalizer of $\Delta(\cA z)$ inside $\Delta(z)(\cM \ovt \rL \Gamma_i) \Delta(z)$ contains $\Delta(\cM z)$. So \cite[Theorem 3.1]{PV12} says that the basic construction von Neumann algebra $z_1 \, \langle \cM \ovt \rL \Gamma_i , e_{\cM \ot 1} \rangle \, z_1$ admits a $\Delta(\cM z)$-central state $\Om$ such that $\Om(\Delta(x)) = \Tr(x)$ for all $x \in \cM z$.

Put $\cC_i := q(\cB \rtimes \Gamma_i)q$. Observe that
$$E_{\cM \ot 1} (\Delta(x)) = \Delta(E_{q \cB q}(x)) \quad\text{for all}\;\; x \in \cC_i \; .$$
Therefore, there is a unique normal $*$-homomorphism
$$\Psi : \langle \cC_i , e_{q \cB q} \rangle \recht \langle \cM \ovt \rL \Gamma_i , e_{\cM \ot 1} \rangle : \begin{cases} \Psi(x) = \Delta(x) \;\;\text{for all}\;\; x \in \cC_i \; , \\
\Psi(e_{q\cB q}) = e_{\cM \ot 1} \; . \end{cases}$$
Define the state $\Om_1$ on $\langle \cC_i , e_{q \cB q} \rangle$ given by $\Om_1(S) = \Om(z_1 \Psi(S) z_1)$. Note that $\Om_1$ is $\cC_i$-central and that $\Om_1(x) = \Tr(xz)$ for all $x \in \cC_i$.

Since $q \cB q$ is amenable, also $\langle \cC_i , e_{q \cB q} \rangle$ is an amenable von Neumann algebra acting on the Hilbert space $\rL^2(\cM)$. So there exists a (non-normal) conditional expectation $\cT : \B(\rL^2(\cM)) \recht \langle \cC_i , e_{q \cB q} \rangle$. The composition $\Om_1 \circ \cT$ is a $\cC_i$-central state on $\B(\rL^2(\cM))$ whose restriction to $\cC_i$ is normal. Therefore $\cC_i$ must have an amenable direct summand. Then also $\cB \rtimes \Gamma_i$ has an amenable direct summand. So, we have shown that $z = 0$.

Choose $\eps > 0$. For every $i = 1,\ldots,n$, we can take a finite subset $\cF_i \subset \Gamma_i$ such that $$\|a - P^i_{\cF_i}(a)\|_2 < \frac{\eps}{n}$$ for all $a$ in the unit ball of $\cA$. The projections $P^i_{\cF_i}$ commute and their product equals the projection $P_\cF$ of $\rL^2(\cM)$ onto the closed linear span of $\{q b u_g q \mid b \in \cB , g \in \cF\}$, where $\cF = \cF_1 \times \cdots \times \cF_n$. It follows that $\|a - P_\cF(a)\|_2 < \eps$ for all $a$ in the unit ball of $\cA$. By Lemma \ref{lem.intertwine-entirely}, this precisely means that $\cA p \preceq_{\cM} q \cB q$ for every nonzero projection $p \in \cA' \cap \cM$.
\end{proof}

Before proving Theorem \ref{thmA}, we need one more technical lemma.

\begin{lem}\label{lem.Fourierzero}
Let $(\cB,\Tr)$ be a von Neumann algebra equipped with a normal semifinite faithful trace $\Tr$. Assume that $\Gamma$ is a countable group and $\Gamma \actson \cB$ a trace preserving action. Put $\cN = \cB \rtimes \Gamma$. Denote by $E : \cN \recht \cB$ the canonical conditional expectation and still denote by $\Tr$ the normal semifinite faithful trace on $\cN$ given by $\Tr \circ E$. For every $w \in \cN$ and $g \in \Gamma$, denote by $(w)_g := E_\cB(w u_g^*)$ the $g$'th Fourier coefficient of $w$.

Assume that $(w_n)$ is a bounded net in $\cN$ such that $(w_n)_g \recht 0$ $\ast$-strongly for all $g \in \Gamma$. Then,
\begin{equation}\label{eq.aim}
\lim_n \|E_\cB(x w_n y^*)\|_{2,\Tr} = 0 \quad\text{for all}\;\; x,y \in \cN \cap \rL^2(\cN,\Tr) \; .
\end{equation}
\end{lem}
\begin{proof}
We may and will assume that $(w_n)$ is a net in the unit ball of $\cN$. Define $W = \lspan \{ a u_g \mid a \in \cB \cap \rL^2(\cB,\Tr) , g \in \Gamma\}$. Observe that $W$ is a $\|\cdot\|_{2,\Tr}$-dense subspace of $\rL^2(\cN,\Tr)$. We claim that
$$\lim_n \|E_\cB( x w_n y^*)\|_{1,\Tr} = 0 \quad\text{for all}\;\; x,y \in W \; .$$
To prove this claim, it suffices to take $x = a u_g$ and $y = b u_h$ with $a,b \in \cB \cap \rL^2(\cB,\Tr)$ and $g,h \in \Gamma$. In that case, we have
$$E_\cB(x w_n y^*) = a \, \sigma_g\bigl((w_n)_{g^{-1} h}\bigr) \, b^* \; .$$
Therefore, by the Cauchy-Schwarz inequality, we have
$$\|E_\cB(x w_n y^*)\|_{1,\Tr} \leq \|a\|_{2,\Tr} \, \bigl\| \sigma_g\bigl((w_n)_{g^{-1} h}\bigr) \, b^* \bigr\|_{2,\Tr} = \|a\|_{2,\Tr} \, \bigl\| (w_n)_{g^{-1} h} \, \sigma_{g^{-1}}(b)^* \bigr\|_{2,\Tr} \; ,$$
which tends to zero because $(w_n)_{g^{-1} h}$ tends to $0$ strongly as a net of operators on $\rL^2(\cB,\Tr)$.

Then choose $x,y \in \cN \cap \rL^2(\cN,\Tr)$ and $\eps > 0$. First take $x_1 \in W$ and then take $y_1 \in W$ such that
$$\|x - x_1\|_{2,\Tr} \, \|y \|_{2,\Tr} < \eps \quad\text{and}\quad \|x_1\|_{2,\Tr} \, \|y - y_1 \|_{2,\Tr} < \eps \; .$$
It follows from the Cauchy-Schwarz inequality and the claim above that
$$
\limsup_n \|E_\cB(x w_n y^*)\|_{1,\Tr} < 2 \eps + \limsup_n \|E_\cB(x_1 \, w_n \, y_1^*)\|_{1,\Tr} = 2\eps \; .
$$
Since this holds for all $\eps > 0$, we have shown that $\|E_\cB(x w_n y^*)\|_{1,\Tr} \recht 0$. Since $E_\cB(x w_n y^*)$ is a bounded net in $\cB$, it also converges to zero in $\|\,\cdot\,\|_{2,\Tr}$. So, \eqref{eq.aim} is proven.
\end{proof}

We are now ready to prove Theorem \ref{thmA}.

\begin{proof}[Proof of Theorem \ref{thmA}]
We put $B = \rL^\infty(X)$ and $M = B \rtimes \Gamma$. We denote by $E_B : M \recht B$ the canonical conditional expectation.
Assume that $A \subset M$ is a Cartan subalgebra that is not unitarily conjugate with $B$. We will prove that $B \rtimes \Gamma_i$ is an amenable von Neumann algebra for at least one $i \in \{1,\ldots,n\}$. By Theorem \ref{cartan}, we get that $A \npreceq_M B$. By Theorem \ref{intertwining-general}, we get a net of unitaries $(v_j)_{j \in J}$ in $A$ such that $\lim_{j \in J} E_B(v_j u_g^*) = 0$ $*$-strongly for all $g \in \Gamma$.

Denote $\cB = \rL^\infty(X \times \R)$ and denote by $\Gamma \actson \cB$ the Maharam extension of $\Gamma \actson B$. We identify $\core(M) = \cB \rtimes \Gamma$ and we denote by $E_\cB : \core(M) \recht \cB$ the canonical conditional expectation. We denote by $\Tr$ the canonical normal semifinite faithful trace on $\core(M)$ and note that $E_\cB$ is trace preserving. Since $\Gamma \actson X$ is essentially free, also $\Gamma \actson X \times \R$ is essentially free and hence, $\cB \subset \core(M)$ is maximal abelian. Since $E_\cB(x) = E_B(x)$ for all $x \in M$, we have that $\lim_{j \in J} E_\cB(v_j u_g^*) = 0$ $*$-strongly for all $g \in \Gamma$.

Choose a faithful normal state $\om$ on $A$ and choose a faithful normal conditional expectation $E_A : M \recht A$. We still denote by $\om$ the extension $\om \circ E_A$ to a faithful normal state on $M$. Since $A \subset M$ is a Cartan subalgebra, we have that $A \rtimes_\om \R \subset M \rtimes_\om \R$ is a regular subalgebra. Obviously, $A \subset M^\om$ and $A \rtimes_\om \R$ is abelian. Denote by $\Pi_\om : M \rtimes_\om \R \recht \core(M)$ the canonical trace preserving $*$-isomorphism given by Connes's Radon-Nikodym cocycle theorem. We have $\Pi_\om(\pi_\om(x)) = x$ for all $x \in M$.

Choose a nonzero finite trace projection $p \in A \rtimes_\om \R$. By Lemma \ref{lem.equiv}, the finite trace projection $\Pi_\om(p) \in \cB \rtimes \Gamma$ is equivalent with a projection $q \in \cB$. Take $v \in \cB \rtimes \Gamma$ such that $v^* v = \Pi_\om(p)$ and $vv^* = q$. Denote $\cM := q(\cB \rtimes \Gamma) q$ and $\cA = v \Pi_\om(A \rtimes_\om \R) v^*$. Since $A \rtimes_\om \R$ is abelian and $A \rtimes_\om \R \subset M \rtimes_\om \R$ is regular, it follows from \cite[Lemma 3.5]{popa-malleable1} that $\cA \subset \cM$ is regular.

Define $w_j := v v_j \Pi_\om(p) v^* = v \Pi_\om( \pi_\om(v_j) p) v^*$. Note that $(w_j)_{j \in J}$ is a net of unitaries in $\cA$. We observed above that $\lim_{j \in J} E_\cB(v_j u_g^*) = 0$ $*$-strongly for all $g \in \Gamma$. So by Lemma \ref{lem.Fourierzero}, we have that $\lim_j \|E_{\cB q}(x w_j y)\|_2 = 0$ for all $x,y \in \cM$. So, $\cA \npreceq_\cM \cB q$. Theorem \ref{newkey} then provides an $i \in \{1,\ldots,n\}$ such that $\cB \rtimes \Gamma_i$ has an amenable direct summand.

It remains to prove that $B \rtimes \Gamma_i$ is amenable. First observe that $\cB \rtimes \Gamma_i$ can be identified with $\core(B \rtimes \Gamma_i)$. So, Takesaki's duality theorem \cite[Theorem X.2.3]{takesakiII} implies that $B \rtimes \Gamma_i$ has an amenable direct summand. Denote by $\Gammah_i$ the direct product of all $\Gamma_j$, $j \neq i$. Then, $\Gammah_i$ naturally acts by automorphisms of $B \rtimes \Gamma_i$. Since $M$ is a factor, this action is ergodic on $\cZ(B \rtimes \Gamma_i)$. Since $B \rtimes \Gamma_i$ has an amenable direct summand, it follows that $B \rtimes \Gamma_i$  is amenable.
\end{proof}

The proof of Corollary \ref{corA} is now immediate.

\begin{proof}[Proof of Corollary \ref{corA}]
Note that $\Gamma = \F_2$ is weakly amenable and belongs to $\QHreg$ (see e.g.\ \cite[Lemma 2.4]{PV12} for references). We claim that $M:= \rL^\infty(X \times Y) \rtimes \Gamma$ is nonamenable. Indeed, $M$ contains $N := \rL^\infty(X \times Y) \rtimes \Ker \pi \cong (\rL^\infty(X) \rtimes \Ker \pi) \ovt \rL^\infty(Y)$ as a von Neumann subalgebra with expectation. Since $\Ker \pi \actson (X,\mu)$ is probability measure preserving and since $\Ker \pi$ is a nonamenable group, it follows that $N$ is nonamenable. So, $M$ follows nonamenable as well.

By Theorem \ref{thmA}, $\rL^\infty(X \times Y)$ is the unique Cartan subalgebra of $M$ up to unitary conjugacy.

It remains to determine the type and the flow of weights of $M$. Put $P = \rL^\infty(Y) \rtimes \Z$. First consider the trivial cases. If $Y$ admits an equivalent $\Z$-invariant probability measure, both $M$ and $P$ are of type \tIIone. If $Y$ admits an equivalent $\Z$-invariant infinite measure, both $M$ and $P$ are of type \tIIinfty\ (because $Y$ was assumed to be nonatomic to rule out the possibility for $P$ to be of type \tIinfty). So, assume that $P$ is of type \tIII. It remains to prove that $M$ and $P$ have an isomorphic flow of weights.

Denote by $\Z \actson \Ytil = Y \times \R$ the Maharam extension of $\Z \actson Y$. Consider the action $\Gamma \actson X \times \Ytil$ given by $g \cdot (x,y) = (g \cdot x,\pi(g) \cdot y)$. Since $\Gamma \actson X$ is measure preserving, the action $\Gamma \actson X \times \Ytil$ can be identified with the Maharam extension of $\Gamma \actson X \times Y$. So, the flow of weights of $M$ can be identified with the natural action of $\R$ on the von Neumann algebra $\rL^\infty(X \times \Ytil)^\Gamma$ of $\Gamma$-invariant functions. Since $\Ker \pi \actson X$ is ergodic, we get that
$$\rL^\infty(X \times \Ytil)^\Gamma = 1 \ot \rL^\infty(\Ytil)^\Z \; .$$
Since the flow of weights of $P$ is given by the natural action of $\R$ on $\rL^\infty(\Ytil)^\Z$, we have found the required isomorphism between the flow of weights of $M$ and $P$.
\end{proof}

\section{Proofs of Theorems~\ref{thmB} and \ref{thmC}}

\begin{proof}[Proof of Theorem~\ref{thmB}]
The proof is very similar to the one of \cite[Theorem 3.2]{CS11}. Put $B = \rL^\infty(X)$ and $M = B \rtimes \Gamma$. Let $e \in M$ be a projection and $P \subset e M e$ a von Neumann subalgebra with expectation $E_P : eMe \recht P$ such that $pPp \npreceq_M B$ for all nonzero projections $p \in P$. Lemma $\ref{intertwining-abelian}$ provides a diffuse abelian $\ast$-subalgebra $A \subset P$ with expectation $E_A : P \recht A$ such that $A \npreceq_M B$.

Choose a faithful normal tracial state $\vphi$ on $A$ and extend $\vphi$ to a normal state on $eMe$ by the formula $\vphi = \vphi \circ E_A \circ E_P$. It then follows that the modular group $(\si_t^\vphi)_{t \in \R}$ globally preserves the subalgebras $A \subset P \subset ePe$ and hence, also globally preserves the subalgebras $P' \cap eMe \subset A' \cap eMe \subset eMe$. It then follows from \cite[Theorem IX.4.2]{takesakiII} that the inclusion $P' \cap eMe \subset A' \cap eMe$ is with expectation. So, in order to show that $P' \cap eMe$ is amenable, it suffices to prove that $A' \cap eMe$ is amenable.

Since $A$ is abelian and $\vphi = \vphi \circ E_A \circ E_P$, we have $A \subset (eMe)^\varphi$, and $A \subset \cZ((A' \cap eMe) \rtimes_\varphi \R)$.
Let $E_B : M \to B$ be the canonical faithful normal conditional expectation. We will denote by $\core(M)$ (resp.\ $\core(B)$) the continuous core of $M$ (resp.\ $B$) and by $\Tr$ the canonical semifinite faithful normal trace on $\core(M)$ so that $\Gamma \curvearrowright \core(B)$ is a trace-preserving action and $\core(M) = \core(B) \rtimes \Gamma$. Using the canonical trace-preserving $\ast$-isomorphism $\Pi_{\varphi} : eMe \rtimes_\varphi \R \to e\core(M)e$ (see Subsection $\ref{flow}$), we write $\core(A' \cap eMe) = \Pi_{\varphi}((A' \cap eMe) \rtimes_\varphi \R)$. Note that $\Tr_{|\core(A' \cap eMe)}$ is semifinite and $A \subset \cZ(\core(A' \cap eMe))$.

Let $p \in \core(A' \cap eMe)$ be a nonzero finite trace projection. Fix a finite set of unitaries $F \subset \cU(p \core(A' \cap eMe)p)$ and a nonzero projection $z \in p\core(A' \cap eMe)p$ that commutes with all $u \in F$. We denote by $J$ the anti-unitary involution on the Hilbert space $\rL^2(\core(M),\Tr)$.
To prove that $p \core(A' \cap eMe)p$ is amenable, we use the criterion given in \cite[Remark 5.29]{connes76} (see \cite[Lemma 2.2]{haagerup} for the non-factorial case). So, we need to show that
\begin{equation}\label{haagerup}
\Bigl\| \sum_{u \in F} u z \otimes J u z J \Bigr\|_\minim = |F| \; ,
\end{equation}
where $\|\,\cdot\,\|_\minim$ denotes the minimal (spatial) C$^*$-tensor norm on $\core(M) \ot_\alg J \core(M) J$. In order to prove \eqref{haagerup}, we use the malleable deformation discovered by Chifan-Sinclair in \cite{CS11}.

Let $\pi : \Gamma \to \cO(H_\R)$ be an orthogonal representation that is weakly contained in the regular representation, together with a proper map $c : \Gamma \to H_\R$ such that $\sup_{x \in \Gamma} \|c(g x h) - \pi_g(c(x))\| < \infty$ for all $g, h \in \Gamma$. Consider the Gaussian construction associated with the orthogonal representation $\pi$. This yields an abelian von Neumann algebra $(D, \tau)$ generated by a family of unitaries $(\omega(\xi))_{\xi \in H_\R}$ that satisfy
\begin{itemize}
\item $\omega(0) = 1$, $\omega(- \xi) = \omega(\xi)^*$, $\omega(\xi + \eta) = \omega(\xi)\omega(\eta)$ for all $\xi, \eta \in H_\R$.
\item $\tau(\omega(\xi)) = \exp(- \|\xi\|^2)$ for all $\xi \in H_\R$.
\end{itemize}
The Gaussian action is then given by $g \cdot (\omega(\xi)) = \omega(\pi_g(\xi))$ for all $g \in \Gamma$ and all $\xi \in H_\R$. Define $\widetilde M = (D \ovt B) \rtimes \Gamma$ and observe that $\core(\widetilde M) = (D \ovt \core(B)) \rtimes \Gamma$. We will regard $\core(M)$ as a unital von Neumann subalgebra of $\core(\widetilde M)$. We still denote by $J$ the canonical anti-unitary involution on $\rL^2(\core(\widetilde M))$, extending the $J$ on $\rL^2(\core(M))$ considered above.

Define the Hilbert spaces
\begin{align*}
\cH & = \rL^2(\core(B)) \otimes \ell^2(\Gamma) = \rL^2(\core(M)), \\
\widetilde {\cH} & = \rL^2(D) \otimes \rL^2(\core(B)) \otimes \ell^2(\Gamma) = \rL^2(\core(\widetilde M)).
\end{align*}
Let $v : \cH \to \widetilde{\cH}$ be the isometry defined by $v(\zeta \otimes \delta_g) = 1 \otimes \zeta \otimes \delta_g$ for all $\zeta \in \rL^2(\core(B))$ and all $g \in \Gamma$. Put $e = vv^*$ and $e^\perp = 1 - e$. Define the $1$-parameter group of unitaries ${(V_t)}_{t \in \R}$ on the Hilbert space $\widetilde{\cH}$ by
$$V_t (d \otimes \zeta \otimes \delta_h) = d \omega(t c(h)) \otimes \zeta \otimes \delta_h \;\; , \quad \forall d \in D, \forall \zeta \in \rL^2(\core(B)), \forall h \in \Gamma, \forall t \in \R.$$

We need the following lemma from \cite{CS11}.

\begin{lem}[{\cite[Lemma 2.6]{CS11}}]\label{uniform-convergence}
For all $x \in \core(B) \rtimes_{\red} \Gamma$, we have that
$$\lim_{t \to 0} \|x V_t v - V_t v x\|_\infty = \lim_{t \to 0} \|J x J V_t v - V_t v J x J\|_\infty = 0.$$
\end{lem}

Since Lemma \ref{uniform-convergence} only holds for $x$ in the C$^*$-algebra $\core(B) \rtimes_{\red} \Gamma$, we have to carefully approximate the elements $uz$, $u \in F$, by elements in this C$^*$-algebra. This approximation is given as follows. Let $\cE$ be the finite dimensional operator space spanned by $\{1\} \cup Fz \cup F^*z$. Since $\Gamma$ is assumed to be an exact group and since $\core(B)$ is abelian, hence exact, the reduced crossed product C$^*$-algebra $\core(B) \rtimes_{\red} \Gamma$ is exact and thus locally reflexive (see \cite[Theorem 9.3.1]{BO08}). Regarding $\cE \subset (\core(B) \rtimes_{\red} \Gamma)^{**}$, there exists a net of completely contractive positive maps $\varphi_i : \cE \to \core(B) \rtimes_{\red} \Gamma$ such that $\varphi_i \to \Id_{\cE}$ pointwise ultraweakly.  By the Hahn-Banach separation theorem, we may further assume that $\varphi_i(u z) \to u z$ strongly, for all $u \in F \cup F^*$. Thus $\lim_i \|(\varphi_i(u z) - u z)p\|_{2, \Tr} = 0$ for all $u \in F \cup F^*$.

Recall that $A \npreceq_M B$. So, Theorem \ref{intertwining-general} provides a sequence of unitaries $w_n \in \cU(A)$ for which the Fourier coefficients $(w_n)_g$ in the crossed product decomposition $M = B \rtimes \Gamma$ tend to zero $\ast$-strongly for all $g \in \Gamma$. Viewing $w_n$ as a sequence in $\core(M) = \core(B) \rtimes \Gamma$, the Fourier coefficients stay the same and hence, still tend to zero $\ast$-strongly. Since $A \subset \cZ(\core(A' \cap eMe))$, we see that the unitaries $w_n$ commute with the projection $z$. So, it follows from Lemma \ref{lem.Fourierzero} that
\begin{equation}\label{eq.convergence}
\lim_n \|E_{\core(B)} ((w_n z)_g) \|_{2, \Tr} = 0 \quad\text{for all}\;\; g \in \Gamma \; .
\end{equation}

For clarity, we denote by $\hat{x} \in \rL^2(\core(M))$ the vector that corresponds to an element $x \in \core(M)$ with $\Tr(x^* x) < \infty$.

\begin{claim}
Put $\delta = \|z\|_{2,\Tr} / 2$. For every $t > 0$, there exists an element  $x_t \in (A z)_1$ such that $\| e^\perp V_t v \widehat{x_t} \|_{\widetilde{\cH}} \geq \delta$.
\end{claim}

{\it Proof of the claim.} Fix $t > 0$. We prove that we can take $x_t = w_n z$ for $n$ large enough. Since $c : \Gamma \to H_\R$ is proper, we can take a finite subset $\cI \subset \Gamma$ such that $1-\exp(-t^2 \|c(g)\|^2) \geq 1/2$ for all $g \in \Gamma - \cI$. By \eqref{eq.convergence}, we can take $n$ large enough such that
$$\sum_{g \in \cI} \|E_{\core(B)} ((w_n z)_g) \|_{2, \Tr}^2 < \frac{1}{2} \|z\|_{2,\Tr}^2  \; .$$
It follows that
$$\sum_{g \in \Gamma - \cI} \|E_{\core(B)} ((w_n z)_g) \|_{2, \Tr}^2 > \frac{1}{2} \|z\|_{2,\Tr}^2 \; .$$
Put $x_t := w_n z$. A direct computation yields
\begin{align*}
\|e^\perp V_t v \, \widehat{w_n z}\|^2_{\widetilde{\cH}} &= \sum_{g \in \Gamma} (1 - \exp(-t^2 \|c(g)\|^2)) \|E_{\core(B)}((w_n z)_g)\|^2_{2, \Tr} \\
&\geq \sum_{g \in \Gamma - \cI} (1 - \exp(-t^2 \|c(g)\|^2)) \|E_{\core(B)}((w_n z)_g)\|^2_{2, \Tr} \\
&\geq \sum_{g \in \Gamma - \cI} \frac{1}{2} \|E_{\core(B)}((w_n z)_g)\|^2_{2, \Tr} > \frac{1}{4} \|z\|_{2,\Tr}^2 \; .
\end{align*}
This proves the claim.

We denote $\xi_t = e^\perp V_t v \widehat{x_t} \in \widetilde{\cH} \ominus \cH$. We now use Lemma \ref{uniform-convergence} and the fact that for all $u \in F$, the elements $uz$ and $w_n$ commute. We also use that $\core(M)$ and $J \core(M) J$ commute with the projection $e$.
So, we find for all $u \in F$ that
\begin{align*}
\limsup_{t \to 0} \|\varphi_i(u z) \, J\varphi_i(u z)J \, \xi_t - \xi_t\|_{\widetilde{\cH}} & = \limsup_{t \to 0} \|e^\perp(\varphi_i(u z) \, J\varphi_i(u z)J \, V_t v \widehat{x_t} - V_t v \widehat{x_t})\|_{\widetilde{\cH}}\\
& \leq  \limsup_{t \to 0} \|\varphi_i(u z) J \varphi_i(u z) JV_t v \widehat{x_t} - V_t v \widehat{x_t}\|_{\widetilde{\cH}} \\
& =  \limsup_{t \to 0} \|V_t v \varphi_i(u z) J \varphi_i(u z) J \widehat{x_t} - V_t v \widehat{x_t}\|_{\widetilde{\cH}} \\
& =  \limsup_{t \to 0} \| \varphi_i(u z) x_t \varphi_i(u z)^* - x_t \|_{2,\Tr} \\
& \leq   \limsup_{t \to 0} \|(uz) x_t (u z)^* - x_t\|_{2,\Tr}  + 2 \|(\varphi_i(u z) - u z) p\|_{2, \Tr} \\
& =  2 \|(\varphi_i(u z) - u z) p\|_{2, \Tr}.
\end{align*}
Since $\xi_t \in \cHtil \ominus \cH$ with $\|\xi_t\|_{\cHtil} \geq \delta$ for all $t > 0$, and since $\|(\varphi_i(u z) - u z) p\|_{2, \Tr} \recht 0$ for all $u \in F$, it follows from the above computation that
$$\limsup_i \Bigl\| \sum_{u \in F} \vphi_i(uz) \, J \vphi_i(uz) J \Bigr\|_{\B(\cHtil \ominus \cH)} \geq |F| \; .$$
Since the representation $\pi$ is weakly contained in the regular representation, a combination of \cite[Lemma 1.7]{anan95} and \cite[Lemma 3.5]{vaes-cohomology} shows that the binormal representation
$$\Theta : \core(M) \otimes_{\alg} J \core(M) J \to  \B(\widetilde{\cH} \ominus \cH) : \Theta(a \ot JbJ) = a \, J b J$$
is continuous with respect to the minimal C$^*$-tensor norm $\|\,\cdot\,\|_\minim$ on $\core(M) \ot_\alg J\core(M)J$. We therefore get that
$$\limsup_i \Bigl\| \sum_{u \in F} \vphi_i(uz) \ot J \vphi_i(uz) J \Bigr\|_\minim \geq \limsup_i \Bigl\| \sum_{u \in F} \Theta(\vphi_i(uz) \ot J \vphi_i(uz) J ) \Bigr\|_{\B(\cHtil \ominus \cH)} \geq |F| \; .$$
The maps $\vphi_i : \cE \recht \core(M)$ and $\overline{\vphi_i} : J \cE J \recht J \core(M) J : \overline{\vphi_i}(J b J) = J \vphi_i(b) J$ are completely contractive. So we conclude that
$$\Bigl\| \sum_{u \in F} uz \ot J uz J \Bigr\|_\minim \geq \limsup_i \Bigl\| \sum_{u \in F} \vphi_i(uz) \ot J \vphi_i(uz) J \Bigr\|_\minim \geq |F| \; .$$
As explained above, this implies that $p \core(A' \cap M)p$ is amenable for all nonzero finite trace projections $p \in \core(A' \cap M)$. Hence $\core(A' \cap M)$ is amenable and the theorem is proven.
\end{proof}

To prove Theorem \ref{thmC}, we need the following equivalence relation version of \cite[Remark 5.29]{connes76} and \cite[Lemma 2.2]{haagerup}. For completeness, we include a proof.

Let $\cR$ be a countable pmp equivalence relation on the standard probability space $(X,\mu)$. Denote by $m$ the measure on $\cR$ given by
$$m(\cU) = \int_X |\{y \in X \mid (x,y) \in \cU\}| \; \rd\mu(x) = \int_X |\{x \in X \mid (x,y) \in \cU\}| \; \rd\mu(y)$$
for all measurable subsets $\cU \subset \cR$. We denote $M = \rL(\cR)$ and identify $\rL^2(M) = \rL^2(\cR,m)$. To every $\psi$ in the full group $[\cR]$ of $\cR$ corresponds a unitary $u(\psi) \in M$ whose action on $\rL^2(\cR,m)$ is given by $(u(\psi)\xi)(x,y) = \xi(\psi^{-1}(x),y)$. We denote by $J$ the canonical anti-unitary involution on $\rL^2(\cR,m)$ given by $(J \xi)(x,y) = \overline{\xi(y,x)}$.

We view $\rL^\infty(\cR)$ as acting on $\rL^2(\cR,m)$ by multiplication operators. Note that the unitaries $u(\psi)$, $\psi \in [\cR]$, normalize $\rL^\infty(\cR)$ and that $\rL^\infty(X) \subset \rL^\infty(\cR)$, by identifying a function $F \in \rL^\infty(X)$ with the function on $\cR$ given by $(x,y) \mapsto F(x)$.
Recall from \cite[Definition 5]{connes-feldman-weiss} that $\cR$ is called \emph{amenable} if there exists a state $\Om$ on $\rL^\infty(\cR)$ satisfying
\begin{align*}
& \Om(F) = \int_X F(x) \; \rd\mu(x) \;\;\text{for all}\;\; F \in \rL^\infty(X) \quad\text{and}\\
& \Om(u(\psi) F u(\psi)^*) = \Om(F) \;\;\text{for all}\;\; \psi \in [\cR], F \in \rL^\infty(\cR) \; .
\end{align*}
By \cite[Theorem 10]{connes-feldman-weiss}, a countable pmp equivalence relation is amenable if and only if it is hyperfinite.

\begin{lem} \label{lem.crit-amen}
The countable pmp equivalence relation $\cR$ is amenable if and only if for all non-negligible $\cR$-invariant measurable subsets $\cU \subset X$ and all $\psi_1,\ldots,\psi_n \in [\cR]$, we have
\begin{equation}\label{eq.haagerup-equiv}
\Bigl\| \sum_{k=1}^n u(\psi_k) \one_\cU \ot J u(\psi_k) \one_\cU J \Bigr\|_\minim = n \; .
\end{equation}
\end{lem}
\begin{proof}
With exactly the same argument as in \cite[Lemma 2.2]{haagerup}, the validity of \eqref{eq.haagerup-equiv} for all $\cR$-invariant measurable subsets $\cU \subset X$ and all $\psi_1,\ldots,\psi_n \in [\cR]$, is equivalent with the existence of a state $\Om$ on $\rL^\infty(\cR)$ satisfying
\begin{align*}
& \Om(F) = \int_X F(x) \; \rd\mu(x) \;\;\text{for all}\;\; F \in \rL^\infty(X)^\cR \quad\text{and}\\
& \Om(u(\psi) F u(\psi)^*) = \Om(F) \;\;\text{for all}\;\; \psi \in [\cR], F \in \rL^\infty(\cR) \; ,
\end{align*}
where we use the notation $\rL^\infty(X)^\cR$ to denote the von Neumann algebra of $\cR$-invariant bounded measurable functions. To prove the lemma, it therefore suffices to show that such a state $\Om$ automatically satisfies $\Om(\one_\cV) = \mu(\cV)$ for all measurable subsets $\cV \subset X$. Denote by $\Psi$ the mean on $X$ given by $\Psi(\cV) = \Om(\one_\cV)$. It follows that $\Psi$ coincides with $\mu$ on the measurable $\cR$-invariant subsets $\cV \subset X$ and that $\Psi$ is invariant under $[\cR]$. We need to prove that $\Psi$ coincides with $\mu$ on all measurable subsets of $X$.

We first prove this statement when $\cR$ is homogeneous of type \tI$_n$, $1 \leq n < \infty$, resp.\ when $\cR$ is homogeneous of type \tIIone.
Assume that $\cR$ is of type \tI$_n$. Then $\cR$ admits a fundamental domain $\cV \subset X$ of measure $1/n$. We can choose $\psi_1,\ldots,\psi_n \in [\cR]$ such that for a.e.\ $x \in \cV$, the equivalence class of $x$ is given by $\{\psi_1(x),\ldots,\psi_n(x)\}$. Note that the subsets $(\psi_k(\cV))_{k=1}^n$ form a partition of $X$ up to measure zero. If $\cU \subset \cV$ is a measurable subset, the sets $\psi_k(\cU)$, $k=1,\ldots,n$, are disjoint and their union is $\cR$-invariant. Hence,
$$n \mu(\cU) = \mu(\cup_{k=1}^n \psi_k(\cU)) = \Psi(\cup_{k=1}^n \psi_k(\cU)) = \sum_{k=1}^n \Psi(\psi_k(\cU)) = n \Psi(\cU) \; .$$
So, $\mu(\cU) = \Psi(\cU)$ for all $\cU \subset \cV$. By $\psi_k$-invariance, also $\mu(\cU) = \Psi(\cU)$ for all $\cU \subset \psi_k(\cV)$. By finite additivity, $\mu(\cU) = \Psi(\cU)$ for all $\cU \subset X$.

Next assume that $\cR$ is of type \tIIone. Denote by $E : \rL^\infty(X) \recht \rL^\infty(X)^\cR$ the trace preserving conditional expectation. Recall that if $\cU,\cV \subset X$ are measurable subsets, then the following two statements are equivalent.
\begin{itemize}
\item There exists a $\psi \in [\cR]$ such that $\psi(\cU) = \cV$.
\item We have $E(\one_\cU) = E(\one_\cV)$.
\end{itemize}
Write $\rL^\infty(X)^\cR = \rL^\infty(Y,\eta)$, where $(Y,\eta)$ is a standard probability space. We can then write $(X,\mu)$ as the direct integral over $(Y,\eta)$ of a measurable field of standard probability spaces. Since $\cR$ is of type \tIIone, almost all these probability spaces are non-atomic and hence isomorphic with the interval $[0,1]$ equipped with the Lebesgue measure. So we find a measure preserving isomorphism of probability spaces $\theta : [0,1] \times Y \recht X$ such that
\begin{itemize}
\item $F(\theta(t,y)) = F(y)$ for all $F \in \rL^\infty(Y) = \rL^\infty(X)^\cR$ and a.e.\ $(t,y) \in [0,1] \times Y$.
\item $(E(F))(y) = \int_0^1 F(\theta(t,y)) \rd t$ for all $F \in \rL^\infty(X)$ and a.e.\ $y \in Y$.
\end{itemize}
Whenever $G : Y \recht [0,1]$ is measurable, we define the measurable set $\cW(G) \subset [0,1] \times Y$ given by
$$\cW(G) = \{(t,y) \mid 0 \leq t \leq G(y) \} \; .$$
We put $\cV(G) = \theta(\cW(G))$. By construction, $E(\one_{\cV(G)}) = G$. Applying this construction to $G = E(\one_\cU)$ for a given measurable subset $\cU \subset X$, it follows that $E(\one_{\cV(G)}) = E(\one_{\cU})$. Hence, there exists a $\psi \in [\cR]$ such that $\psi(\cU) = \cV(G)$. Since both $\mu$ and $\Psi$ are invariant under $[\cR]$, it suffices to prove that $\mu$ and $\Psi$ coincide on the sets of the form $\cV(G)$ for any measurable function $G : Y \recht [0,1]$.

Choose such a measurable function $G : Y \recht [0,1]$ and fix an integer $n \geq 1$. Put $\cU = \cV(G)$. We will prove that $|\mu(\cU) - \Psi(\cU)| \leq 1/n$. Define the following partition of $Y$ into measurable subsets $(Y_k)_{k=1}^n$.
\begin{align*}
Y_k &= \Bigl\{y \in Y \;\Big|\; \frac{k-1}{n} \leq G(y) < \frac{k}{n} \Bigr\} \;\;\text{for}\;\; 1 \leq k \leq n-1 \;\;\text{and}\\
Y_n &= \Bigl\{y \in Y \;\Big|\; \frac{n-1}{n} \leq G(y) \leq 1 \Bigr\} \; .
\end{align*}
Put $\cU_k := \theta([0,1] \times Y_k)$ and note that $(\cU_k)_{k=1}^n$ is a partition of $X$ into $\cR$-invariant subsets. By finite additivity, it suffices to prove that $|\mu(\cU \cap \cU_k) - \Psi(\cU \cap \cU_k)| \leq \mu(\cU_k) / n$ for all $k=1,\ldots,n$. Fix $k \in \{1,\ldots,n\}$.

Define the partition $(\cV_i)_{i=1}^n$ of $\cU_k$ into the measurable subsets given by
$$\cV_i = \theta\Bigl( \Bigl[ \frac{i-1}{n},\frac{i}{n}\Bigr) \times Y_k \Bigr) \;\;\text{for $1 \leq i \leq n-1$, and}\;\; \cV_n = \theta\Bigl( \Bigl[ \frac{n-1}{n},1\Bigr] \times Y_k \Bigr) \; .$$
By construction, we have
\begin{equation} \label{eq.tussenstap}
\cup_{i=1}^{k-1} \cV_i \subset \cU \cap \cU_k \subset \cup_{i=1}^k \cV_i \; .
\end{equation}
By construction, we also have that $E(\one_{\cV_i}) = \frac{1}{n} \one_{\cU_k}$ for all $i=1,\ldots,n$. So, we can take $\psi_1,\ldots,\psi_n \in [\cR]$ such that $\psi_i(\cV_1) = \cV_i$, up to measure zero. We get that
$$n \mu(\cV_1) = \mu(\cup_{i=1}^n \cV_i) = \mu(\cU_k) = \Psi(\cU_k) = \sum_{i=1}^n \Psi(\psi_i(\cV_1)) = n \Psi(\cV_1) \; ,$$
so that $\mu(\cV_1) = \Psi(\cV_1) = \mu(\cU_k) / n$ and hence $\mu(\cV_i) = \Psi(\cV_i) = \mu(\cU_k)/n$ for all $i$. It then follows from \eqref{eq.tussenstap} that both $\mu(\cU \cap \cU_k)$ and $\Psi(\cU \cap \cU_k)$ lie between $(k-1) \mu(\cU_k) / n$ and $k \mu(\cU_k) / n$. This proves the required inequality $|\mu(\cU \cap \cU_k) - \Psi(\cU \cap \cU_k)| \leq \mu(\cU_k) / n$.
This concludes the proof in the case where $\cR$ is of type \tIIone.

Finally, partition $X$ into $\cR$-invariant measurable subsets $\cV$ and $(\cV_n)_{n=1}^\infty$ such that $\cR_{|\cV}$ is of type \tIIone\ and $\cR_{|\cV_n}$ is of type \tI$_n$. By the paragraphs above, we know that $\mu(\cU) = \Psi(\cU)$ if $\cU$ is a measurable subset of one of the $\cV_n, \cV$. By finite additivity, the same holds if $\cU$ is a measurable subset of a finite union of $\cV_n,\cV$. Choose $\eps > 0$. Take $n_0$ such that
$$\mu(\cup_{n=n_0+1}^\infty \cV_n) < \eps \; .$$
Since $\mu$ and $\Psi$ coincide on the $\cR$-invariant subsets, also $\Psi(\cup_{n=n_0+1}^\infty \cV_n) < \eps$. Put $\cW = \cV \cup \cup_{n=1}^{n_0} \cV_n$. For every measurable $\cU \subset X$, we have that $|\mu(\cU) - \mu(\cU \cap \cW)| < \eps$ and $|\Psi(\cU) - \Psi(\cU \cap \cW)| < \eps$. Since $\mu(\cU \cap \cW) = \Psi(\cU \cap \cW)$, we get that $|\mu(\cU) - \Psi(\cU)| < 2 \eps$ for all $\eps > 0$ and all measurable subsets $\cU \subset X$. So $\Psi = \mu$ and this concludes the proof of the lemma.
\end{proof}

\begin{proof}[Proof of Theorem~\ref{thmC}]
Fix a nonsingular free action $\Gamma \curvearrowright (X, \mu)$ and a non-negligible subset $\cV \subset X$. We denote by $\cR = \cR(\Gamma \actson X)$ the orbit equivalence relation and we write
$$M = \rL^\infty(X) \rtimes \Gamma = \rL(\cR) \; .$$
Assume that we can decompose $\cV = X_1 \times X_2$ and that we are given recurrent nonsingular equivalence relations $\cS_i$ on $X_i$ such that $\cS_1 \times \cS_2 \subset \cR_{|\cV}$. We must prove that the $\cS_i$ are amenable. By symmetry, it suffices to prove that $\cS_2$ is amenable.

Write $P_i = \rL(\cS_i)$ and $e = \one_\cV$. We have $P_1 \ovt P_2 \subset eMe$ and this inclusion is with expectation, since every inclusion coming from a subequivalence relation is with expectation. In particular, $P_i \subset eMe$ is with expectation for every $i=1,2$. Since $\cS_1$ is recurrent, the von Neumann algebra $P_1$ has no type \tI\ direct summand. Since $\rL^\infty(X)$ is of type \tI, it follows that $p P_1 p \npreceq_M \rL^\infty(X)$ for every nonzero projection $p \in P_1$.

If $\Gamma$ would be exact, it would now follow immediately from Theorem \ref{thmB} that $P_2 \subset P_1' \cap M$ is amenable. We now no longer assume that $\Gamma$ is exact and we therefore only get the following. Consider the Maharam extension $\Gamma \actson X \times \R$ and denote the associated orbit equivalence relation by $\core(\cR) = \cR(\Gamma \actson X \times \R)$. We equip $X \times \R$ with the canonical infinite $\Gamma$-invariant measure. Then $\core(\cR)$ is a \tIIinfty\ relation and we canonically have $\core(\cS_2) \subset e \core(\cR) e$. We also have $\core(M) = \rL(\core(\cR))$ and $\core(P_2) = \rL(\core(\cS_2))$.

Choose an arbitrary finite measure subset $\cU \subset \cV \times \R$ and denote by $p = \one_\cU$ the corresponding projection in $\rL^\infty(X \times \R)$.
We consider the restricted equivalence relation ${\cS_2}_{|\cU}$ and its full group $\cG := [{\cS_2}_{|\cU}]$. For every $\psi \in \cG$, we have a canonical unitary $u(\psi) \in \core(p P_2 p)$. To prove the amenability of ${\cS_2}_{|\cU}$, it suffices, by Lemma \ref{lem.crit-amen}, to prove that for every projection $z \in \cZ(pP_2 p)$ and all $\psi_1,\ldots,\psi_n \in \cG$, we have
\begin{equation}\label{eq.equiv-aim}
\Bigl\| \sum_{k=1}^n u(\psi_k)z \ot J u(\psi_k)z J \Bigr\|_\minim = n \; .
\end{equation}
Denote by $\cE \subset P_2$ the finite dimensional operator space spanned by
$$\{1\} \cup \{u(\psi_k) z \mid k =1,\ldots,n\} \cup \{u(\psi_k)^* z \mid k =1,\ldots,n\} \; .$$
Below we construct a concrete sequence of completely positive contractive maps $\vphi_i : \cE \recht \rL^\infty(X \times \R) \rtimes_\red \Gamma$ such that $\vphi_i(x)p \recht xp$ $*$-strongly for all $x \in \cE$.

Note that the equality \eqref{eq.equiv-aim} is a special case of the equality \eqref{haagerup} that we established in the proof of Theorem \ref{thmB}. Indeed, in the proof of Theorem \ref{thmB} we showed that \eqref{haagerup} holds for all finite trace projections $p \in \core(A' \cap eMe)$, all finite sets of unitaries $F \subset \cU(p \core(A' \cap eMe)p)$ and all projections $z \in p \core(A' \cap eMe) p$ that commute with $F$. In the proof of \eqref{haagerup}, we only used the exactness of $\Gamma$ at one specific place, namely to ensure the a priori existence of  such a sequence of completely positive contractive maps $\vphi_i$, through local reflexivity. So, to conclude the proof of Theorem \ref{thmC}, it suffices to concretely construct the sequence $\vphi_i$.

For every $k = 1,\ldots,n$, we find partitions $(\cU^k_g)_{g \in \Gamma}$ and $(\cV^k_g)_{g \in \Gamma}$ of $\cU$ of such that $\psi_k(y) = g \cdot y$ for a.e.\ $y \in \cU^k_g$ and $\psi_k^{-1}(y) = g \cdot y$ for a.e.\ $y \in \cV^k_g$. Refining these partitions, we can choose an increasing sequence of projections $p_i \in \rL^\infty(\cU)$ such that $p_i \recht p$ strongly and such that both $u(\psi_k) p_i$ and $u(\psi_k)^* p_i$ belong to $\rL^\infty(X \times \R) \rtimes_\alg \Gamma$ for all $k = 1,\ldots,n$ and all $i$. It now suffices to define $\vphi_i(x) = p_i x p_i$.
\end{proof}

\section{Amalgamated free products and HNN extensions}\label{amalgamation}

\subsection{Von Neumann algebra inclusions without trivial corner}

\begin{df}
We say that  an inclusion of von Neumann algebras $A \subset P$ has a \emph{trivial corner} if there exists a nonzero projection $p \in A' \cap P$ such that $Ap = pPp$. If no such projection exists, we say that the inclusion $A \subset P$ has no trivial corner.
\end{df}

We first prove the following elementary lemma.

\begin{lem}\label{lem.triv-corner}
Let $A \subset P$ be an inclusion of von Neumann algebras.
\begin{enumerate}
\item If $A \subset P$ has no trivial corner and $p \in A' \cap P$ and $e \in A$ are projections such that $r = ep$ is nonzero, then the inclusion $rAr \subset rPr$ has no trivial corner either.
\item Let $p_i \in A' \cap P$ and $e_i \in A$ be projections such that the projections $r_i = e_i p_i$ are nonzero and satisfy $\sum_i r_i = 1$. Then $A \subset P$ has a trivial corner if and only if there exists an $i$ such that $r_i A r_i \subset r_i P r_i$ has a trivial corner.
\item Assume that $A$ is purely atomic. Then $A \subset P$ has no trivial corner if and only if $P$ is diffuse.
\item Assume that $A \subset P$ is a maximal abelian subalgebra with expectation. Then $A \subset P$ has a trivial corner if and only if $P$ has a type \tI\ direct summand.
\end{enumerate}
\end{lem}
\begin{proof}
(1) It follows immediately from the definition that $Ap \subset pPp$ has no trivial corner. So, we only need to prove that $eAe \subset ePe$ has no trivial corner whenever $A \subset P$ has no trivial corner and $e \in A$ is a nonzero projection. So assume that $q \in A'e \cap ePe$ is a nonzero projection such that $eAeq = q P q$. We will prove that $A \subset P$ has a trivial corner. Denote by $q_1$ the smallest projection in $A' \cap P$ such that $q \leq q_1$. Note that $q = e q_1$. We then get
$$A q A \; P \; A q A = A \; q P q \; A = A \; e A eq \; A = A \; e A eq_1 \; A \subset A q_1 \; .$$
It follows that $q_1 P q_1 \subset A q_1$ and hence $q_1 P q_1 = A q_1$. So, $A \subset P$ has a trivial corner.

(2) If one of the $r_i A r_i \subset r_i P r_i$ has a trivial corner, it follows from (1) that $A \subset P$ has a trivial corner. Conversely assume that $A \subset P$ has a trivial corner, i.e.\ $Ap = p Pp$ for a nonzero projection $p \in A' \cap P$. Take $i$ such that $r_i p \neq 0$. Denote by $q_i$, resp.\ $v_i \in A' \cap P$ the left support projection, resp.\ polar part of $p_i p$. Since $r_i p \neq 0$, we have $e_i q_i \neq 0$.
From the assumption that $Ap = pPp$, it follows that
$$A q_i = v_i \; Ap \; v_i^* = v_i P v_i^* = q_i P q_i \; .$$
Multiplying on the left and on the right by $e_i$, we get that $e_i A e_i q_i = e_i q_i P e_i q_i$, implying that $r_i A r_i \subset r_i P r_i$ has a trivial corner.

(3) Writing $1$ as a sum of minimal projections in $A$, (3) follows immediately from (2), by noticing that $\C 1 \subset P$ has no trivial corner if and only if $P$ is diffuse.

(4) If $A \subset P$ has a trivial corner, we find a nonzero projection $p \in A' \cap P = A$ such that $Ap = pPp$. So, $p$ is a nonzero abelian projection in $P$ and therefore $P$ must have a type \tI\ direct summand. Conversely, assume that $z \in \cZ(P)$ is a nonzero central projection such that $Pz$ is of type \tI. Since $Az \subset Pz$ is a maximal abelian subalgebra with expectation, Lemma \ref{lem.masa-typeI} provides a nonzero projection $p \in Az$ such that $p$ is abelian in $Pz$. In particular, $pPp = Ap$, so that $A \subset P$ has a trivial corner.
\end{proof}

The following lemma follows immediately from the classification of measure preserving factor maps $(X,\mu) \recht (Y,\eta)$ between standard probability spaces, as proven in \cite[Section 4]{Ro47}. Indeed, the corresponding inclusion $\rL^\infty(Y) \subset \rL^\infty(X)$ has no trivial corner if and only if the disintegration of $\mu$ w.r.t.\ $(Y,\eta)$ is almost everywhere nonatomic, which precisely means that there exists a measure preserving isomorphism $(X,\mu) \cong (Y \times [0,1],\nu \times \text{\rm Lebesgue})$ that respects the factor maps onto $(Y,\eta)$.

\begin{lem}\label{lem.abelian-triv-corner}
Let $(B,\tau)$ be an abelian von Neumann algebra with separable predual and faithful normal trace $\tau$. Let $A \subset B$ be a von Neumann subalgebra. Then the following two statements are equivalent.
\begin{itemize}
\item The inclusion $A \subset B$ has no trivial corner.
\item There exists a trace preserving $*$-isomorphism $\pi : B \recht A \ovt \rL \Z$ such that $\pi(A) = A \ot 1$.
\end{itemize}
\end{lem}

Let $A \subset P$ be an inclusion of finite von Neumann algebras. It is easy to see that $P \preceq_P A$ if and only if there exists a nonzero projection $p \in A' \cap P$ such that $Ap \subset pPp$ has finite Jones index \cite{jones-index, pimsner-popa}. Note however that an inclusion $A \subset P$ can be of finite index and yet without trivial corner. For this it suffices to consider an outer action of a finite group $\Lambda$ on the hyperfinite \tIIone\ factor $R$, in which case $R \subset R \rtimes \Lambda$ is a finite index inclusion that is irreducible and hence, without trivial corner.

When $A \subset P$ is an inclusion of semifinite von Neumann algebras with $A$ being of type \tI, the situation is much more clear as we explain now.

\begin{lem}\label{triv-corner-I}
Let $(P,\Tr)$ be a von Neumann algebra with separable predual and a distinguished normal semifinite faithful trace $\Tr$. Assume that $A$ is a type \tI\ von Neumann subalgebra of $P$ such that $\Tr_{|A}$ is semifinite. Denote by $E_A : P \recht A$ the unique trace preserving conditional expectation. Then the following statements are equivalent.
\begin{enumerate}
\item The inclusion $A \subset P$ has no trivial corner.
\item Whenever $e \in P$ and $q \in A$ are projections such that $\Tr(e) < \infty$ and $0 < q \leq e$, then $ePe \npreceq_{ePe} qAq$.
\item Whenever $q \in A$ and $p \in A'q \cap qPq$ are nonzero projections such that $\Tr(q) < \infty$, there exists a unitary $u \in \cU(pPp)$ such that $E_{qAqp}(u^n) = 0$ for all $n \in \Z \setminus \{0\}$.
\end{enumerate}
\end{lem}
\begin{proof}
We will prove that $3 \Rightarrow 2 \Rightarrow 1 \Rightarrow 3$.

$3 \Rightarrow 2$. Fix projections $e \in P$ and $q \in A$ such that $\Tr(e) < \infty$ and $0 < q \leq e$. To prove that $ePe \npreceq_{ePe} qAq$, we choose a nonzero $ePe$-$qAq$-subbimodule $\cH \subset e \rL^2(P) q$ and we must prove that $\cH$ is of infinite dimension as a right $qAq$-module. Note that $\cH = e \rL^2(P) p$ for some nonzero projection $p \in A'q \cap qPq$. Denote by $z \in \cZ(qAq)$ the support of $E_{qAq}(p)$.
Because (3) holds, we can take a unitary $u \in \cU(pPp)$ such that $E_{qAqp}(u^n) = 0$ for all $n \in \Z \setminus \{0\}$. Define $\cK_n \subset e\rL^2(P)p$ as the $\|\,\cdot\,\|_2$-closure of $\{u^n a \mid a \in qAq\}$. By construction, the $\cK_n$ are orthogonal right $qAq$-submodules of $\cH$ with $\dim_A(\cH_n) = \tau(z)$. Hence, $\cH$ is indeed of infinite dimension as a right $qAq$-module.

$2 \Rightarrow 1$. Assume that $A \subset P$ has a trivial corner. Take a projection $r \in A' \cap P$ such that $rPr = Ar$. Since $\Tr_{|A}$ is semifinite, we can take a projection $e \in A$ such that $\Tr(e) < \infty$ and $e r \neq 0$. By construction, the inclusion $e A e \subset e P e$ has a trivial corner. In particular $ePe \preceq_{ePe} eAe$. So (2) does not hold for the projections $e \in P$ and $q = e$.

$1 \Rightarrow 3$. Fix nonzero projections $q \in A$ and $p \in A'q \cap qPq$ such that $\Tr(q) < \infty$. Since $A \subset P$ has no trivial corner, it follows from Lemma \ref{lem.triv-corner} that also $qAq p \subset p Pp$ has no trivial corner. So, replacing $P$ by $pPp$ and $A$ by $qAq p$, we may assume that $\Tr$ is a finite trace and we only need to prove the existence of a unitary $u \in \cU(P)$ such that
\begin{equation}\label{eq.mijndoel}
E_A(u^n) = 0 \;\;\text{for all}\;\; n \in \Z \setminus \{0\} \; .
\end{equation}
We may further assume that $A$ is abelian. Indeed, since $A$ is of type \tI, we may find an orthogonal family $(e_k)_k$ of abelian projections in $A$ such that $\sum_k e_k = 1$. Assume moreover that we have found unitaries  $u_k \in \cU(e_k P e_k)$ satisfying \eqref{eq.mijndoel} for the inclusion $e_k A e_k \subset e_k P e_k$. Then the unitary $u = \sum_k u_k \in \cU(P)$ satisfies \eqref{eq.mijndoel} for the inclusion $A \subset P$.

From now on, we assume that $A$ is abelian. We decompose $P = P z_{{\rm I}} \oplus P z_{{\rm II}}$, where $P z_{{\rm I}}$ is of type \tI\ and $P z_{{\rm II}}$ is of type \tIIone. So it suffices to prove the existence of $u \in \cU(P)$ satisfying \eqref{eq.mijndoel} separately in the cases where $P$ is of type \tI, resp.\ of type \tIIone. We first assume that $P$ is of type \tI. Choose $A \subset B \subset P$ such that $B$ is maximal abelian in $P$. We claim that the inclusion $A \subset B$ has no trivial corner. Indeed, if $q \in B$ is a nonzero projection and $Aq = Bq$, it follows that $Aq \subset q P q$ is maximal abelian. Since $qPq$ is of type \tI\ and $Aq \subset qPq$ has no trivial corner, this contradicts Lemma \ref{lem.triv-corner}. So the claim is proven. Lemma \ref{lem.abelian-triv-corner} then provides a unitary $u \in B$ such that \eqref{eq.mijndoel} holds.

Finally assume that $P$ is of type \tIIone. Choose again $A \subset B \subset P$ such that $B$ is maximal abelian in $P$. It suffices to find a unitary $u \in \cU(P)$ such that $E_B(u^n) = 0$ for all $n \in \Z \setminus \{0\}$. Since $P$ is of type \tIIone, there exists a projection $q \in P$ such that $q$ and $1-q$ are equivalent projections. By \cite[Corollary F.8]{BO08}, the projection $q$ is equivalent with a projection $p \in B$. Since $P$ is finite, also $1-q$ is equivalent with $1-p$ and we conclude that $p$ is equivalent with $1-p$.

Since $p$ and $1-p$ are equivalent projections in $P$, we can take a self-adjoint unitary $u_1 \in \cU(P)$ such that $u_1 p u_1^* = 1-p$. Note that automatically $E_B(u_1) = 0$. We can repeat the same reasoning for the maximal abelian subalgebra $Bp$ of the type \tIIone\ von Neumann algebra $p P p$, yielding a projection $q \in Bp$ and a self-adjoint unitary $v \in \cU(pPp)$ such that $v q v^* = p - q$. Define the self-adjoint unitary $u_2 \in \cU(P)$ given by $u_2 := v + u_1 v u_1$. Note that $u_1 u_2 = u_2 u_1$ and that $u_1,u_2$ generate a copy of $\rL(\Z/2\Z \times \Z/2\Z)$ inside $P$ with the property that $E_B(x) = \tau(x) 1$ for all $x \in \rL(\Z/2\Z \times \Z/2\Z)$.

We continue the same construction inductively. Writing $G = \bigoplus_\N \Z/2\Z$, we find a copy of $\rL(G)$ inside $P$ such that $E_B(x) = \tau(x)1$ for all $x \in \rL(G)$. Since $\rL(G)$ is a diffuse abelian von Neumann algebra, we can take a unitary $u \in \rL(G)$ such that $\tau(u^n) = 0$ for all $n \in \Z \setminus \{0\}$. So, again \eqref{eq.mijndoel} holds.
\end{proof}

We now prove that under the right assumptions, the property of having no trivial corner is preserved under taking the continuous core.

\begin{prop}\label{core}
Let $A \subset P$ be an inclusion of von Neumann algebras without trivial corner and with expectation. Assume that $A$ is of type \tI. Then the inclusion of continuous cores $\core(A) \subset \core(P)$ has no trivial corner either.
\end{prop}

\begin{proof}
Since the continuous core of a type \tI\ von Neumann algebra is of type \tI\ and the one of a von Neumann algebra of type \tII\ or \tIII\ is always of type \tII, the only nontrivial case to consider, is the one where $P$ is a type \tI\ von Neumann algebra.

Let $(p_i)_{i \in J}$ be a family of abelian projections in the type \tI\ von Neumann algebra $A$, with $\sum_i p_i = 1$. The inclusion $\core(p_i A p_i) \subset \core(p_i P p_i)$ can be identified with the inclusion $p_i \core(A) p_i \subset p_i \core(P) p_i$. By Lemma \ref{lem.triv-corner}, it suffices to prove that for all $i \in J$, the inclusion $\core(p_i A p_i) \subset \core(p_i P p_i)$ has no trivial corner.

Fix $i \in J$. Replacing $A$ by $p_i A p_i$ and replacing $P$ by $p_i P p_i$, we may therefore assume that $A$ is abelian. We claim that $A \subset A' \cap P$ has no trivial corner. Otherwise, we would find a nonzero projection $p \in A' \cap P$ such that $Ap = p(A' \cap P)p$, which would mean that $Ap \subset pPp$ is maximal abelian and hence would contradict Lemma \ref{lem.triv-corner}.

Since $A \subset P$ is with expectation, also $A' \cap P \subset P$ is with expectation. The inclusion $\cZ(A' \cap P) \subset A' \cap P$ is with expectation and hence, also the inclusion $\cZ(A' \cap P) \subset P$ is with expectation. We then canonically get
$$\core(A) \subset \core(\cZ(A' \cap P)) \subset \core(P)$$
and it suffices to prove that $\core(A) \subset \core(\cZ(A' \cap P))$ has no trivial corner. Since $\cZ(A' \cap P)$ is abelian, this last inclusion equals $A \ovt \rL(\R) \subset \cZ(A' \cap P) \ovt \rL(\R)$. It therefore suffices to prove that $A \subset \cZ(A' \cap P)$ has no trivial corner.
Assume the contrary and take a nonzero projection $p \in \cZ(A' \cap P)$ such that $Ap = \cZ(A' \cap P)p$. Since $A$ is abelian and $A \subset P$ is with expectation, it follows from Lemma \ref{lem.masa-typeI} that $A' \cap P$ is of type \tI. Therefore, $A' \cap P$ admits an abelian projection $e$ with $ep \neq 0$. Then $e(A' \cap P)e = \cZ(A' \cap P)e$ and we conclude that $A ep = ep P ep$, contradicting the assumption that $A \subset P$ has no trivial corner.
\end{proof}

\subsection{Applications to amalgamated free products von Neumann algebras}

For $i = 1, 2$, let $B \subset M_i$ be an inclusion of von Neumann algebras with expectation $E_i : M_i \to B$. Recall that the \emph{amalgamated free product} $(M, E) = (M_1, E_1) \ast_ B (M_2, E_2)$ is the von Neumann algebra $M$ generated by $M_1$ and $M_2$ where the faithful normal conditional expectation $E : M \to B$ satisfies the freeness condition:
$$E(x_1 \cdots x_n) = 0 \mbox{ whenever } x_j \in M_{\iota_j} \ominus B \mbox{ and } \iota_j \neq \iota_{j + 1} \; .$$
Here and in what follows, we denote by $M_i \ominus B$ the kernel of the conditional expectation $E_i : M_i \recht B$.
We refer to \cite{voiculescu85,voiculescu92,ueda-pacific} for more details on the construction of amalgamated free products in the framework of von Neumann algebras.

We always assume that $B$ is a $\sigma$-finite von Neumann algebra. Fix a faithful normal state $\varphi$ on $B$. We still denote by $\varphi$ the faithful normal state $\varphi \circ E$ on $M$. We realize the continuous core of $M$ as $\core(M) = M \rtimes_\vphi \R$. Recall from \cite{ueda-pacific} that $\si_t^\vphi(M_i) = M_i$ and hence
$$\core(M) = \core(M_1) \ast_{\core(B)} \core(M_2) \; .$$

In \cite[Theorem 5.2]{CH08}, the following statement is made: if $B$ is of type \tI, if $M_1 \neq B \neq M_2$ and if $M$ is a nonamenable factor, then $M$ is prime. As such, this statement is wrong. As we already mentioned in the introduction, the mere condition $M_1 \neq B \neq M_2$ is too weak to avoid ``trivialities'' in which a corner of $M$ equals a corner of one of the $M_i$'s.

\begin{example}
We provide two types of examples where amalgamated free product factors $M=M_1 *_B M_2$ essentially are an amplification of one of the $M_i$.

{\bf (1)} Let $\Lambda \curvearrowright (Y, \nu)$ be a free ergodic pmp action of a nontrivial countable group $\Lambda$ on a standard probability space $(Y,\nu)$. Put $\Gamma = \Lambda \ast \Z$ and consider the \emph{induced} action $\Gamma \curvearrowright (X, \mu)$, with $X = \Ind_\Lambda^\Gamma Y$. Write $M = \rL^\infty(X) \rtimes \Gamma$. Since $\Gamma = \Lambda \ast \Z$, we canonically have $M = M_1 *_B M_2$, where $B = \rL^\infty(X)$, $M_1 = B \rtimes \Lambda$ and $M_2 = B \rtimes \Z$. So, $M_1 \neq B \neq M_2$. On the other hand, we also canonically have
$$M = (\rL^\infty(Y) \rtimes \Lambda) \ovt \B(\ell^2(\Gamma / \Lambda)) \; .$$
Taking a nonamenable group $\Lambda$ and a free ergodic pmp action $\Lambda \actson (Y,\nu)$ such that $\rL^\infty(Y) \rtimes \Lambda$ is not prime (e.g.\ McDuff), it follows that $M$ is a nonamenable factor that is not prime.

Also note that the induction construction comes with a $\Lambda$-invariant finite trace projection $p \in B$ such that $p (B \rtimes \Gamma) p = Bp \rtimes \Lambda$ and hence,
$$p (M_1 *_B M_2) p = p M_1 p \; .$$

{\bf (2)} State-preserving actions can only be induced from finite index subgroups. So the previous example does not provide natural examples of trivialized amalgamated free products of type \tIIone. But they exist as well, as demonstrated by the following example. Let $P$ be a \tIIone\ factor and $p_n, q_n$ nonzero projections in $P$ such that $\sum_{n = 0}^\infty (p_n + q_n) = 1$. Define the von Neumann algebras
\begin{align*}
M_1 & =  p_0 P p_0 \oplus \bigoplus_{n = 0}^\infty (p_{n + 1} + q_n) P (p_{n + 1} + q_n) \; ,\\
M_2 & = \bigoplus_{n = 0}^\infty (p_n + q_n) P (p_n + q_n) \; ,\\
B & = \bigoplus_{n = 0}^\infty (p_n P p_n \oplus q_n P q_n) \; .
\end{align*}
Taking the amalgamated free product w.r.t.\ the unique trace preserving conditional expectations, one checks that $p_0(M_1 \ast_B M_2)p_0 = p_0 M_1 p_0 = p_0 B p_0$.
\end{example}

The following statement provides a correct version of \cite[Theorem 5.2]{CH08}. The formulation is rather technical, but in combination with Theorem \ref{amalgamated} below, we will deduce the general primeness Theorem \ref{thmD} stated in the introduction.

If $i = 1$, resp.\ $2$, we denote $i' := 2$, resp.\ $1$.

\begin{theo}[Erratum to {\cite[Theorem 5.2]{CH08}}]\label{prime}
Let $M_i$, $i=1,2$, be von Neumann algebras with separable predual. Assume that $B$ is a type \tI\ von Neumann algebra and that we are given inclusions $B \subset M_i$ with faithful normal conditional expectations. Consider $M = M_1 \ast_B M_2$ and assume that $M$ is a nonamenable factor.

If $M$ is not prime, there exist $i \in \{1,2\}$, and projections $q \in \cZ(\core(M_i))$ and $p \in \cZ(\core(B))$ such that $0 < q \leq p$ and
$$q \core(M) q = \core(M_i) q \quad\text{and}\quad p \core(M_{i'}) p = \core(B) p \; .$$
\end{theo}

\begin{proof}
Write $\cM = \core(M)$, $\cM_i = \core(M_i)$ and $\cB = \core(B)$. We canonically have $\cM = \cM_1 *_\cB \cM_2$.
Assume that $M$ is not prime. The proof of \cite[Theorem 5.2]{CH08} is correct up to the point where an $i \in \{1,2\}$ and a nonzero projection $q \in \cM_i^\infty := \cM_i \ovt \B(\ell^2)$ are found such that $q \cM^\infty q = q \cM_i^\infty q$. Assume that $i = 1$.
Inside $\cM_1 \ovt \B(\ell^2)$, the projections $q$ and $1 \ot e_{00}$ have nonzero equivalent subprojections. So there exists a nonzero partial isometry $v \in \cM_1^\infty$ such that $vv^* \leq q$ and $v^*v = q' \ot e_{00}$. Replacing $q$ by $q'$, we have found a nonzero projection $q \in \cM_1$ such that $q \cM q = q \cM_1 q$. Multiplying left and right with $\cM_1$, we conclude that the same equality remains true if we replace $q$ by its central support in $\cM_1$. So we have found a nonzero projection $q \in \cZ(\cM_1)$ such that $q \cM q = \cM_1 q$.

For all $x \in \cM$, we have $qxq \in \cM_1$ and hence $qxq = E_{\cM_1}(qxq) = q E_{\cM_1}(x) q$. We conclude that $q (\cM \ominus \cM_1) q = 0$. Since $\cM_2 \ominus \cB$ is a subset of $\cM \ominus \cM_1$, we get that $q (\cM_2 \ominus \cB) q = 0$. We now put $b = E_\cB(q)$ and $q_0 = q - b$, so that $q_0 \in \cM_1 \ominus \cB$. For every $x \in \cM_2 \ominus \cB$, we get that
$$0 = q x q = q_0 x q_0 + (bx) q_0 + q_0 (xb) + b x b \; .$$
The four terms on the right hand side respectively belong to $(\cM_1 \ominus \cB)(\cM_2 \ominus \cB)(\cM_1 \ominus \cB)$, $(\cM_2 \ominus \cB)(\cM_1 \ominus \cB)$, $(\cM_1 \ominus \cB)(\cM_2 \ominus \cB)$ and $\cM_2 \ominus \cB$. By freeness with amalgamation, these four subspaces are orthogonal. So it follows that $bxb = 0$ for all $x \in \cM_2 \ominus \cB$. Denoting by $p$ the support projection of $b$, we get that $pxp = 0$ for all $x \in \cM_2 \ominus \cB$. Since $q \in \cZ(M_1)$, we get that $b \in \cZ(\cB)$ and hence also $p \in \cZ(\cB)$. We conclude that $p \cM_2 p = \cB p$. By construction $q \leq p$.
\end{proof}

We prove the following rather general result, ensuring that an amalgamated free product $M_1 *_B M_2$ over a type \tI\ von Neumann algebra $B$ often is a nonamenable factor. In \cite[Theorem 4.3 and Theorem 4.13]{ueda12}, Ueda independently proved the factoriality of $M_1 *_B M_2$ under a weaker set of conditions on $B \subset M_i$. Since our method is more elementary than the approach of
\cite{ueda12}, we give a complete proof here as well.

\begin{theo}\label{amalgamated}
Let $M_i$, $i=1,2$, be von Neumann algebras with separable predual. Assume that $B$ is a type \tI\ von Neumann algebra and that we are given inclusions $B \subset M_i$ with faithful normal conditional expectations. Assume that the inclusions $B \subset M_i$ have no trivial corner. Then $M = M_1 \ast_B M_2$ has no amenable direct summand and
$$\cZ(M) = \cZ(M_1) \cap \cZ(B) \cap \cZ(M_2) \quad\text{and}\quad \cZ(\core(M)) = \cZ(\core(M_1)) \cap \cZ(\core(B)) \cap \cZ(\core(M_2)) \; .$$
\end{theo}

\begin{proof}
Write $\cM = \core(M)$, $\cM_i = \core(M_i)$ and $\cB = \core(B)$. We first prove that $\cZ(\cM) = \cZ(\cM_1) \cap \cZ(\cM_2)$. The inclusion $\cZ(\cM_1) \cap \cZ(\cM_2) \subset \cZ(\cM)$ being trivial, take $z \in \cZ(\cM)$. Fix a nonzero finite trace projection $p \in \cM_1$. We prove that $zp \in \cM_1$. Once this is proven for all finite trace projections $p \in \cM_1$, it follows that $z \in \cZ(\cM_1)$. By symmetry, we then also have $z \in \cZ(\cM_2)$.

Let $q \in \cB$ be an arbitrary finite trace projection and put $e = q \vee p$. Since the inclusion $B \subset M_i$ has no trivial corner, the inclusion $\core(B) \subset \core(M_1)$ has no trivial corner either by Proposition \ref{core}. So Proposition \ref{triv-corner-I} implies that $e \cM_1 e \npreceq_{e \cM_1 e} q \cB q$. Then also $p \cM_1 p \npreceq_{e \cM_1 e} q \cB q$.
Since this holds for all finite trace projections $q \in \cB$, it follows from \cite[Lemma 2.2]{houdayer-ricard} that there exists a net of unitaries $v_k \in p \cM_1 p$ such that $\|E_1(x v_k y)\|_{2,\Tr} \recht 0$ for all $x,y \in \cM_1$. Take
$$x = x_0 y_1 x_1 \cdots y_n x_n \quad\text{and}\quad a = a_0 b_1 a_1 \cdots b_m a_m$$
with $n,m \geq 1$, $x_0,x_n,a_0,a_m \in \cM_1$, $x_i, a_j \in \cM_1 \ominus \cB$ for $1 \leq i \leq n-1$, $1 \leq j \leq m-1$ and $y_i,b_i \in \cM_2 \ominus \cB$ for $1 \leq i \leq n$, $1 \leq j \leq m$. Recall that $\cM = \cM_1 *_\cB \cM_2$, w.r.t.\ the natural conditional expectations $E_i : \cM_i \recht \cB$.
Denote by $E : \cM \recht \cB$ the corresponding conditional expectation on the amalgamated free product $\cM = \cM_1 *_\cB \cM_2$. Freeness with amalgamation ensures that
$$E(v_k^* \, x \, v_k \, a) = E(v_k^* \; x_0 y_1 x_1 \cdots y_n \; E_1(x_n \, v_k \, a_0) \; b_1 a_1 \cdots b_m a_m) \; .$$
So, $\lim_k \|E(v_k^* \, x \, v_k \, a)\|_{2,\Tr} = 0$ for all the above choices of $x$ and $a$.
Since $E(v_k^* \, x \, v_k \, a) = E(v_k^* \, xp \, v_k \, pa)$ and since the elements $xp$ and $pa$ span $\|\,\cdot\,\|_{2,\Tr}$-dense subspaces of $(\cM \ominus \cM_1)p$, resp.\ $p (\cM \ominus \cM_1)$, we conclude that $\|E(v_k^* x v_k y)\|_{2,\Tr} \recht 0$ for all $x,y \in \cM \ominus \cM_1$. Put $x = zp - E_1(zp)$ and $y = x^*$. Note that $x$ commutes with all the $v_k \in \cU(p \cM_1 p)$. It follows that $\|E(xx^*)\|_{2,\Tr} = 0$ and hence $x = 0$. So, $zp \in \cM_1$. As explained in the first paragraph of the proof, we have shown that $\cZ(\cM) = \cZ(\cM_1) \cap \cZ(\cM_2)$.

Since $\cZ(M) = \cZ(\core(M)) \cap M$, while $\cZ(\core(M_i)) \cap M = \cZ(M_i)$, we also find that $\cZ(M) = \cZ(M_1) \cap \cZ(M_2)$.

It remains to prove that $M$ has no amenable direct summand. Fix an arbitrary nonzero finite trace projection $p \in \cB$. Observe that
$$p \cM_1 p \ast_{p \cB p} p \cM_2 p \subset p \cM p \; .$$
By Proposition \ref{triv-corner-I}, we can find unitaries $u_i \in \cU(p \cM_i p)$ such that $E_{p \cB p}(u_i^n) = 0$ for all $n \in \Z \setminus \{0\}$. In this way, we obtain a trace-preserving $\ast$-embedding $\rL(\F_2) \hookrightarrow p \cM p$. It follows that for all finite trace projections $p \in \cB$, the finite von Neumann algebra $p \cM p$ has no amenable direct summand. Therefore also $\cM$ has no amenable direct summand.
\end{proof}

Recall that for any factor $M$, the restriction of the dual action $\theta : \R \actson \core(M)$ to the center $\cZ(M)$ is called the \emph{flow of weights} of $M$.

\begin{cor}
Let $B \subset M_i$ and $E_i : M_i \recht B$ be as in Theorem \ref{amalgamated}. Assume that $\cZ(M_1) \cap \cZ(B) \cap \cZ(M_2) = \C 1$ and put $M = M_1 *_B M_2$.
Then the amalgamated free product $M$ is a nonamenable factor and the following holds.
\begin{enumerate}
\item $M$ is of type \tIIone\ if and only if $B$ admits a faithful normal tracial state $\tau$ such that for $i=1,2$, the state $\tau \circ E_i$ is tracial on $M_i$.
\item $M$ is of type \tIIinfty\ if and only if $B$ admits a normal semifinite faithful trace $\Tr$ with $\Tr(1) = +\infty$ such that for $i=1,2$, the normal semifinite weight $\Tr \circ E_i$ is tracial on $M_i$.
\item Fix any normal semifinite faithful weight $\vphi$ on $B$ and put $\vphi_i := \vphi \circ E_i$. The $T$-invariant of $M$ is given by
$$T(M) = \{t \in \R \mid \exists u \in \cU(B) , \forall i \in \{1,2\}, \forall x \in M_i , \si_t^{\vphi_i}(x) = u x u^* \} \; .$$
\item Write $\cZ(B) = \rL^\infty(X)$ and define the flow $\R \actson X \times \R$ given by $t \cdot (x,s) = (x, t+s)$. Identify $\rL^\infty(X \times \R) = \cZ(\core(B))$ and define the factor flows $X \times \R \recht X_i$ such that $\rL^\infty(X_i) = \cZ(\core(B)) \cap \cZ(\core(M_i))$. Then the flow of weights $\R \actson Y$ of $M$ is the unique largest common factor flow $X_i \recht Y$ such that the following diagram commutes.
    $$\begin{matrix}
    & \hspace{-1.5ex}\raisebox{-1ex}[0ex][0ex]{$\nearrow$} & \hspace{-1.5ex} X_1 & \hspace{-1.5ex}\raisebox{-1ex}[0ex][0ex]{$\searrow$} & \\
 X \times \R & & & & \hspace{-1.5ex}Y \\
    & \hspace{-1.5ex}\raisebox{1ex}[0ex][0ex]{$\searrow$} & \hspace{-1.5ex} X_2 & \hspace{-1.5ex}\raisebox{1ex}[0ex][0ex]{$\nearrow$} &
    \end{matrix} \;\;.$$
\end{enumerate}
\end{cor}
\begin{proof}
Choose a normal semifinite faithful trace $\Tr$ on $B$. Define the weights $\vphi_i = \Tr \circ E_i$ on $M_i$. We concretely realize $\core(M_i) = M_i \rtimes_{\vphi_i} \R$ using the modular group of $\vphi_i$ and denote by $(\lambda_s)_{s \in \R} \in \core(M_i)$ the canonical unitaries that implement $(\si_s^{\vphi_i})$.

It follows from Theorem \ref{amalgamated} that $M$ is a nonamenable factor. In particular, $M$ is not of type \tI.

First assume that $M$ is semifinite. So, the flow of weights $\theta : \R \actson \cZ(\core(M))$ is isomorphic with the translation action of $\R$ on $\rL^\infty(\R)$. This means that $\cZ(\core(M))$ is generated by a $1$-parameter group of unitaries $(V_s)_{s \in \R}$ satisfying $\theta_t(V_s) = \exp(ist) V_s$ for all $s, t \in \R$. By Theorem \ref{amalgamated}, we have $\cZ(\core(M)) \subset \cZ(\core(B))$. Since $B$ is semifinite, we have $\cZ(\core(B)) = \cZ(B) \ovt \rL(\R)$. So, $V_s$ must be of the form $V_s = \Delta^{is} \lambda_s$, where $\Delta$ is a nonsingular positive operator affiliated with $\cZ(B)$. Replacing $\Tr$ by the normal semifinite faithful trace given by $\Tr(\Delta \, \cdot)$, we get that $V_s = \lambda_s$, meaning that $\cZ(\core(M))$ is generated by the canonical unitaries $(\lambda_s)_{s \in \R}$. This precisely means that the weights $\Tr \circ E_i$ on $M_i$ and the weight $\Tr \circ E$ on $M$ are tracial.

If $M$ was of type \tIIone, we have $(\Tr \circ E)(1) < \infty$ and hence $\Tr(1) < \infty$. Putting $\tau = \Tr(1)^{-1} \Tr$, we have found a faithful normal tracial state $\tau$ on $B$ satisfying the conditions in (1). If $M$ was of type \tIIinfty, we have $\Tr(1) = +\infty$ and we have found a trace $\Tr$ on $B$ satisfying the conditions in (2).

Conversely, if $\tau$ satisfies the assumptions in (1), then the state $\tau \circ E$  is tracial on $M$ and we conclude that $M$ is of type \tIIone. We similarly get that $M$ is of type \tIIinfty\ if $\Tr$ satisfies the assumptions in (2).

To prove (3), we may assume that $\vphi = \Tr$, where $\Tr$ is a normal semifinite faithful trace on $B$. Put $\vphi_i = \Tr \circ E_i$. If $t \in \R$ and $u \in \cU(B)$ such that $\si_t^{\vphi_i} = \Ad u$ for all $i \in \{1,2\}$, it follows that $\si_t^{\Tr \circ E} = \Ad u$ and hence, $t \in T(M)$. Conversely, assume that $t \in T(M)$.
As above, denote by $\theta : \R \actson \core(M)$ the dual action. Since $t \in T(M)$, \cite[Theorem XII.1.6]{takesakiII} provides a unitary $v \in \cZ(\core(M))$ such that $\theta_s(v) = \exp(its) v$ for all $s \in \R$. By Theorem \ref{amalgamated}, we have that $v \in \cZ(\core(B)) = \cZ(B) \ovt \rL(\R)$. So, $v = u^* \lambda_t$ for some $u \in \cU(\cZ(B))$. Since $v \in \cZ(\core(M))$, we get that $v$ commutes with $M_i$ for both $i=1,2$. This precisely means that $\si_t^{\vphi_i} = \Ad u$ for all $i \in \{1,2\}$.

Finally, statement (4) is just a rephrasing of the equality $\cZ(\core(M)) = \cZ(\core(M_1)) \cap \cZ(\core(B)) \cap \cZ(\core(M_2))$ proved in Theorem \ref{amalgamated}.
\end{proof}

We are now ready to prove Theorem \ref{thmD} and its Corollary \ref{corE}.

\begin{proof}[Proof of Theorem~\ref{thmD}]
By Theorem \ref{amalgamated}, $M$ is a nonamenable factor. Since the inclusion $B \subset M_i$ has no trivial corner, the inclusion $\core(B) \subset \core(M_i)$ has no trivial corner either. So $M$ is prime by Theorem \ref{prime}.
\end{proof}

\begin{proof}[Proof of Corollary~\ref{corE}]
In both cases, $\Gamma \actson (X,\mu)$ is free and ergodic and hence, $M = \rL^\infty(X) \rtimes \Gamma$ is a factor. It suffices to prove that in both cases, $M$ is a nonamenable prime factor, since then $\cR(\Gamma \actson X)$ must be indecomposable as well.

{\bf (1)} Put $M_i = \rL^\infty(X) \rtimes \Gamma_i$ and $B = \rL^\infty(X) \rtimes \Sigma$. We canonically have that $M = M_1 *_B M_2$.
Since $\Gamma_i \curvearrowright (X, \mu)$ is recurrent, $M_i = \rL^\infty(X) \rtimes \Gamma_i$ has no type \tI\ direct summand. Since $\Sigma$ is a finite group, $B = \rL^\infty(X) \rtimes \Sigma$ is a type \tI\ von Neumann algebra. In particular, $B \subset M_i$ has no trivial corner. Theorem \ref{thmD} then says that $M$ is a nonamenable prime factor.

{\bf (2)} We already observed that $M$ is a factor. To see that $M$ is nonamenable, define the subgroups $\Gamma_i < \Gamma$ given by $\Gamma_1 = H$ and $\Gamma_2 = t^{-1} H t$. Note that $\Gamma_1 \cap \Gamma_2 = \Sigma$ and that, in this way, $\Lambda := \Gamma_1 *_\Sigma \Gamma_2 < \Gamma$. Since the action $H \actson (X,\mu)$ is recurrent, also the actions $\Gamma_i \actson (X,\mu)$ are recurrent. Reasoning as in the proof of (1), it follows from Theorem \ref{amalgamated} that $\rL^\infty(X) \rtimes \Lambda$ has no amenable direct summand. Since there is a natural faithful normal conditional expectation $M \recht \rL^\infty(X) \rtimes \Lambda$, it follows that $M$ is nonamenable as well.

To prove that $M$ is prime, we use the HNN construction in the framework of von Neumann algebras and we refer to \cite{ueda-jfa,ueda-illinois,fima-vaes} for more details. Write $N = \rL^\infty(X) \rtimes \Sigma$, $P = \rL^\infty(X) \rtimes H$, $\core(N) = \rL^\infty(X \times \R) \rtimes \Sigma$ and $\core(P) = \rL^\infty(X \times \R) \rtimes H$.  Define the normal $\ast$-homomorphism $\theta : N \hookrightarrow P$, by $\theta(f u_g) = f \circ t^{-1} u_{\theta(g)}$ for all $f \in \rL^\infty(X)$. Moreover define the normal $\ast$-homomorphisms
$$N \oplus N \hookrightarrow \M_2(P), \mbox{ by } x \oplus y \mapsto \begin{pmatrix}
x & 0 \\
0 & \theta(y)
\end{pmatrix}$$
and
$$N \oplus N \hookrightarrow \M_2(N), \mbox{ by } x \oplus y \mapsto \begin{pmatrix}
x & 0 \\
0 & y
\end{pmatrix}.$$
Likewise, define the $\Tr$-preserving normal $\ast$-homomorphisms $\core(\theta) : \core(N) \hookrightarrow \core(P)$, $\core(N) \oplus \core(N) \hookrightarrow \M_2(\core(P))$ and $\core(N) \oplus \core(N) \hookrightarrow \M_2(\core(N))$. Denote by $(e_{ij})$ the canonical matrix units in $\M_2(\core(P))$. It follows from \cite[Proposition 3.1]{ueda-jfa}, that we may identify $M$ with the corner of the amalgamated free product
$$M = e_{11} \left( \M_2(P) \ast_{N \oplus N} \M_2(N) \right) e_{11} \; .$$
Observe that we cannot directly use Theorem~\ref{thmD} to get that $M$ is prime, as the inclusion $N \oplus N \subset \M_2(N)$ has a trivial corner. We already showed however that $M$ is a nonamenable factor. Therefore, also $Q := \M_2(P) *_{N \oplus N} \M_2(N)$ is a nonamenable factor. We will use Theorem \ref{prime} to prove that $Q$ is prime. Then also $M = e_{11} Q e_{11}$ follows prime.

Observe that
$$\core(Q) = \M_2(\core(P)) \ast_{\core(N) \oplus \core(N)} \M_2(\core(N)) \; .$$
If $Q$ would not be prime, Theorem \ref{prime} implies that we are in one of the following situations.
\begin{enumerate}
\item There exists a nonzero projection $p \in \cZ(\core(P))= \cZ(\M_2(\core(P)))$ such that $p \core(Q) p = \M_2(\core(P))p$.
\item There exists a nonzero projection $q \in \cZ(\core(N)) \oplus \cZ(\core(N))$ such that $q \M_2(\core(P)) q = (\core(N) \oplus \core(N))q$.
\end{enumerate}
Since $H \actson (X,\mu)$ is recurrent, we know that $P$ and $\core(P)$ have no type \tI\ direct summand. On the other hand, since $\Sigma$ is a finite group, $\core(N)$ is of type \tI. So, a projection $q$ as in (2) does not exist.

It remains to rule out the existence of a projection $p$ as in (1). Multiplying the equality $p \core(Q) p = \M_2(\core(P)) p$ on the left and on the right by $e_{11}$, it follows that $p \core(M) p = \core(P) p$. Write $\core(M) = \rL^\infty(X \times \R) \rtimes \Gamma$, where $\Gamma \actson X \times \R$ is the Maharam extension of $\Gamma \actson X$. In this picture, $\core(P)$ corresponds to $\rL^\infty(X \times \R) \rtimes H$. Since the action $H \actson X \times \R$ is essentially free, we have $\cZ(\core(P)) \subset \rL^\infty(X \times \R)$. So, $p = \one_\cU$ for a non-negligible $H$-invariant measurable subset $\cU \subset X \times \R$. Since $p \core(M) p = \core(P) p$, we have $pxp = 0$ whenever $x \in \core(M)$ and $E_{\core(P)}(x) =0$. We get in particular that $p u_g p = 0$ for all $g \in \Gamma - H$. This means that the sets $(g \cdot \cU)_{g \in \Gamma / H}$ are all disjoint, up to measure zero. Since the cosets $h t^{-1} H$, $h \in H - \Sigma$, are all different from $t^{-1} H$, it follows that with $\cV = t^{-1} \cdot \cU$, we have that $h \cdot \cV \cap \cV$ has measure zero for every $h \in H - \Sigma$. Removing a set of measure zero from $\cV$, we find a non-negligible subset $\cW \subset \cV$ such that $h \cdot \cW \cap \cW = \emptyset$ for all $h \in H - \Sigma$. Since the action $H \actson X$ is assumed to be recurrent and since $\Sigma$ is finite, we arrived at a contradiction.
\end{proof}

\end{document}